\documentclass[a4paper]{amsart}

\usepackage[utf8]{inputenc}
\usepackage[T1]{fontenc}
\usepackage{lmodern}
\usepackage{amsmath}
\usepackage{amsfonts}
\usepackage{amssymb}
\usepackage{amsthm}
\usepackage{thmtools}
\usepackage{mathtools}
\usepackage{tikz-cd}
\usepackage{multirow}
\usepackage{placeins}
\usepackage{enumitem}
\usepackage[pdftex,
            pdfauthor={Gerald Höhn and Sven Möller},
            pdftitle={Systematic Orbifold Constructions of Schellekens' Vertex Operator Algebras from Niemeier Lattices},
            pdfcreator={pdfTeX}]{hyperref}

\theoremstyle{plain}
\newtheorem{thm}{Theorem}[section]
\newtheorem{prop}[thm]{Proposition}
\newtheorem{cor}[thm]{Corollary}
\newtheorem{lem}[thm]{Lemma}

\newtheorem*{thm*}{Theorem}
\newtheorem*{conj*}{Conjecture}
\newtheorem*{prop*}{Proposition}

\newtheoremstyle{narrow}
  {.5em} 
  {.5em} 
  {\itshape} 
  {} 
  {\bfseries} 
  {.} 
  {.5em} 
  {} 

\theoremstyle{narrow}

\newtheorem*{thn*}{Theorem}

\newtheorem*{conjn*}{Conjecture}

\theoremstyle{definition}
\newtheorem{defi}[thm]{Definition}

\newtheorem*{nota*}{Notation}
\newtheorem{rem}[thm]{Remark}

\newcommand{\Q}{\mathbb{Q}}
\newcommand{\Z}{\mathbb{Z}}
\newcommand{\Ns}{\mathbb{Z}_{>0}}
\newcommand{\N}{\mathbb{Z}_{\geq0}}
\newcommand{\C}{\mathbb{C}}

\renewcommand{\H}{\mathbb{H}}

\renewcommand{\i}{\mathrm{i}}
\newcommand{\e}{\mathrm{e}}

\newcommand{\Aut}{\operatorname{Aut}}
\newcommand{\Hom}{\operatorname{Hom}}

\newcommand{\rk}{\operatorname{rk}}

\newcommand{\voa}{vertex operator algebra}

\newcommand{\VOA}{Vertex Operator Algebra}
\newcommand{\vosa}{vertex operator subalgebra}

\newcommand{\fpvosa}{fixed-point vertex operator subalgebra}

\newcommand{\gdh}{generalised deep hole}

\newcommand{\GDH}{Generalised Deep Hole}
\newcommand{\vac}{\textbf{1}}

\newcommand{\id}{\operatorname{id}}
\newcommand{\amgis}{\zeta}
\newcommand{\eps}{\varepsilon}

\newcommand{\ee}{\mathfrak{e}}
\newcommand{\g}{\mathfrak{g}}
\newcommand{\hh}{\mathcal{H}}
\newcommand{\h}{\mathfrak{h}}
\newcommand{\ad}{\operatorname{ad}}
\newcommand{\Inn}{\operatorname{Inn}}
\newcommand{\Out}{\operatorname{Out}}

\newcommand{\orb}{\operatorname{orb}}

\newcommand{\strathol}{strongly rational, holomorphic}
\newcommand{\strat}{strongly rational}
\newcommand{\II}{I\!I}
\newcommand{\spn}{\operatorname{span}}
\newcommand{\Co}{\operatorname{Co}}
\renewcommand{\O}{\operatorname{O}}
\newcommand{\Com}{\operatorname{Com}}
\newcommand{\good}{short}
\newcommand{\Good}{Short}

\newcommand{\sAA}{1^{24}}
\newcommand{\sBB}{1^82^8}
\newcommand{\sCC}{1^63^6}
\newcommand{\sDD}{2^{12}}
\newcommand{\sEE}{1^42^24^4}
\newcommand{\sFF}{1^45^4}
\newcommand{\sGG}{1^22^23^26^2}
\newcommand{\sHH}{1^37^3}
\newcommand{\sII}{1^22^14^18^2}
\newcommand{\sJJ}{2^36^3}
\newcommand{\sKK}{2^210^2}

\newcommand{\gAA}{\II_{24,0}}
\newcommand{\gBB}{\II_{16,0}(2_{\II}^{+10})}
\newcommand{\gCC}{\II_{12,0}(3^{-8})}
\newcommand{\gDD}{\II_{12,0}(2_{\II}^{-10}4_{\II}^{-2})}
\newcommand{\gEE}{\II_{10,0}(2_{2}^{+2}4_{\II}^{+6})}
\newcommand{\gFF}{\II_{8,0}(5^{+6})}
\newcommand{\gGG}{\II_{8,0}(2_{\II}^{+6}3^{-6})}
\newcommand{\gHH}{\II_{6,0}(7^{-5})}
\newcommand{\gII}{\II_{6,0}(2_{5}^{-1}4_{1}^{+1}8_{\II}^{-4})}
\newcommand{\gJJ}{\II_{6,0}(2_{\II}^{+4}4_{\II}^{-2}3^{+5})}
\newcommand{\gKK}{\II_{4,0}(2_{\II}^{-2}4_{\II}^{-2}5^{+4})}

\newlength{\myl}
\newcommand{\s}{\hspace{\myl}}

\renewcommand{\arraystretch}{1.2}

\begin{document}
\title[]{Systematic Orbifold Constructions of Schellekens' Vertex Operator Algebras from Niemeier Lattices}
\author[Gerald Höhn and Sven Möller]{Gerald Höhn\textsuperscript{\lowercase{a}} and Sven Möller\textsuperscript{\lowercase{b},\lowercase{c}}}
\thanks{\textsuperscript{a}{Kansas State University, Manhattan, KS, United States of America}}
\thanks{\textsuperscript{b}{Rutgers University, Piscataway, NJ, United States of America}}
\thanks{\textsuperscript{c}{Research Institute for Mathematical Sciences, Kyoto University, Kyoto, Japan}}
\thanks{Email: \href{mailto:gerald@monstrous-moonshine.de}{\nolinkurl{gerald@monstrous-moonshine.de}}, \href{mailto:math@moeller-sven.de}{\nolinkurl{math@moeller-sven.de}}}

\begin{abstract}
We present a systematic, rigorous construction of all $70$ \strathol{} \voa{}s $V$ of central charge $24$ with non-zero weight-one space $V_1$ as cyclic orbifold constructions associated with the $24$ Niemeier lattice \voa{}s $V_N$ and certain $226$ \good{} automorphisms in $\Aut(V_N)$.

We show that up to algebraic conjugacy these automorphisms are exactly the \gdh{}s, as introduced in \cite{MS23}, of the Niemeier lattice \voa{}s with the additional property that their orders are equal to those of the corresponding outer automorphisms.

Together with the constructions in \cite{Hoe17} and \cite{MS23} this gives three different uniform constructions of these \voa{}s, which are related through $11$ algebraic conjugacy classes in $\Co_0$.

Finally, by considering the inverse orbifold constructions associated with the $226$ \good{} automorphisms, we give the first systematic proof of the result that each \strathol{} \voa{} $V$ of central charge $24$ with non-zero weight-one space $V_1$ is uniquely determined by the Lie algebra structure of $V_1$.
\end{abstract}

\maketitle

\setcounter{tocdepth}{1}
\tableofcontents
\setcounter{tocdepth}{2}


\section{Introduction}\label{sec:intro}
The programme to classify the \strathol{} \voa{}s of central charge $24$ was initiated by Schellekens in $1993$. He showed that the weight-one subspace $V_1$ of such a \voa{} $V$ is one of $71$ reductive Lie algebras called Schellekens' list \cite{Sch93} (see also \cite{DM04,DM06b,EMS20a}). He conjectured that all potential Lie algebras are realised and that the $V_1$-structure fixes the \voa{} $V$ up to isomorphism.

By contributions of many authors over the last three decades the following classification result is now proved:
\begin{thn*}
Up to isomorphism there are exactly $70$ \strathol{} \voa{}s $V$ of central charge $24$ with $V_1\neq\{0\}$. Such a \voa{} is uniquely determined by its $V_1$-structure.
\end{thn*}
The original proof relies mainly on cyclic orbifold constructions but is based on a case-by-case analysis with a variety of different approaches.

\medskip

In \cite{Hoe17}, a uniform proof of the existence part of the theorem was given, depending on a conjecture on orbifolds of lattice \voa{}s proved in~\cite{Lam20}. Each \voa{} $V$ is realised as a simple-current extension of a certain dual pair in $V$.

Another systematic proof of the existence part of the theorem was recently given in \cite{MS23} (see also \cite{CLM22}) by considering orbifold constructions $V_\Lambda^{\orb(g)}$ associated with \gdh{}s $g$ of the Leech lattice \voa{} $V_\Lambda$. The corresponding inverse orbifold constructions are described in \cite{ELMS21} and are used to give a simpler proof of Schellekens' list of $71$ Lie algebras.

In this work we shall describe a third systematic construction of the \strathol{} \voa{}s $V$ of central charge $24$ with $V_1\neq\{0\}$, namely as orbifold constructions starting from the \voa{}s $V_N$ associated with the $24$ Niemeier lattices $N$, the positive-definite, even, unimodular lattices of rank $24$, which include the Leech lattice~$\Lambda$.

At the centre of these three constructions, as was first observed in \cite{Hoe17}, are $11$ algebraic conjugacy classes (i.e.\ conjugacy classes of cyclic subgroups, see \cite{CCNPW85}) in the isometry group $\O(\Lambda)\cong\Co_0$ of the Leech lattice~$\Lambda$, namely those uniquely specified by the Frame shapes $\sAA$, $\sBB$, $\sCC$, $\sDD$, $\sEE$, $\sFF$, $\sGG$, $\sHH$, $\sII$, $\sJJ$ and $\sKK$ (see \autoref{table:11}). Evidently, these Frame shapes have only non-negative exponents, but they are not characterised in $\Co_0$ by this property (see \autoref{table:25} in the appendix).

\medskip

Given a Niemeier lattice $N$, the outer automorphism group $\Aut(V_N)/K$ of the corresponding lattice \voa{} $V_N$ is isomorphic to $H=\O(N)/W$ where $W$ is the Weyl group of $N$, and $H$ can be embedded into $\O(N)$.

Let $g$ be an automorphism of finite order $n$ of a Niemeier lattice \voa{} $V_N$. Up to conjugation $g$ is of the form $\hat{\nu}\,\e^{-(2\pi\i)h(0)}$ with $\nu\in H$ and $h\in\pi_\nu(N\otimes_\Z\Q)$ where $\pi_\nu$ is the projection onto the elements that are fixed by $\nu$ (see \autoref{thm:latconj}). Then $g$ is called \emph{\good{}} (see \autoref{defi:good}) if
\begin{enumerate}
\item $g$ has type~$0$ (so that the cyclic orbifold construction $V_N^{\orb(g)}$ exists),
\item $\nu$ (i.e.\ the projection of $g$ to $\Aut(V_N)/K\cong H$) has order $n$ and
\item $h\bmod{(N^\nu)'}$ has order $n$.
\end{enumerate}
As the main result of the present text we establish the existence part of the classification theorem in a systematic way:
\begin{thn*}[\autoref{thm:main}, \autoref{cor:main}]
The cyclic orbifold constructions $V_N^{\orb(g)}$, where $N$ runs through the $24$ Niemeier lattices and $g$ through the \good{} automorphisms of the corresponding lattice \voa{}s $V_N$, realise all $70$ non-zero Lie algebras $\g$ on Schellekens' list as weight-one spaces $(V_N^{\orb(g)})_1\cong\g$.
\end{thn*}

We classify the \good{} automorphisms:
\begin{thn*}[\autoref{prop:226}]
There are exactly $226$ algebraic conjugacy classes of \good{} automorphisms of the Niemeier lattice \voa{}s $V_N$, listed in \autoref{table:11x24}. The Frame shapes of their projections to $\Aut(V_N)/K\cong H$ are given by the $11$ Frame shapes in \autoref{table:11}.
\end{thn*}

We then show that the \good{} automorphisms $g$ of the Niemeier lattice \voa{}s $V_N$ are all \gdh{}s as introduced in \cite{MS23}, i.e.
\begin{enumerate}
\item\label{item:prop1} $g$ has type~$0$, 
\item\label{item:prop2} the dimension of $(V_N^{\orb(g)})_1$ attains the upper bound provided by the dimension formula in \cite{MS23} (see \autoref{thm:dimform}), and
\item\label{item:prop3} the weight-one Lie algebras satisfy the orbifold rank condition $\rk((V_N^g)_1)=\rk((V_N^{\orb(g)})_1)$, 
\end{enumerate}
and additionally
\begin{enumerate}[resume]
\item\label{item:prop4} the order of $g$ equals the order of the projection of $g$ to $\Aut(V_N)/K\cong H$.
\end{enumerate}
Conditions \eqref{item:prop2} to \eqref{item:prop4} should be understood as extremality requirements.

The last condition \eqref{item:prop4} entails that the orders of the automorphisms of $V_N$ that we consider are relatively small, namely equal to the orders of the corresponding $11$ Frame shapes listed above. We contrast this to the uniform construction in \cite{MS23}, which uses \gdh{}s of only the Leech lattice \voa{} $V_\Lambda$ projecting to the same $11$ Frame shapes in $\O(\Lambda)$ but with orders up to $46$.

We then prove that the $226$ \good{} automorphisms are already characterised by these four properties:
\begin{thn*}[\autoref{thm:class}]
The \good{} automorphisms of the Niemeier lattice \voa{}s $V_N$ are exactly the automorphisms of $V_N$ satisfying \eqref{item:prop1} to \eqref{item:prop4} (and with $\rk((V_\Lambda^g)_1)>0$ in the case of the Leech lattice $\Lambda$).
\end{thn*}

Finally, we give the first uniform proof of the uniqueness statement in the classification theorem at the beginning of the introduction:
\begin{thn*}[\autoref{thm:unique}]
Let $\g$ be a non-zero Lie algebra on Schellekens' list. Then there is a Niemeier lattice $N$ and a \good{} automorphism $g\in\Aut(V_N)$ such that any \strathol{} \voa{} $V$ of central charge $24$ with $V_1\cong\g$ satisfies $V\cong V_N^{\orb(g)}$. In particular, the \voa{} structure of $V$ is uniquely determined by the Lie algebra structure of $V_1$.
\end{thn*}
The proof follows the strategy laid out in \cite{LS19} and uses the inverse orbifold constructions associated with certain $157$ of the $226$ \good{} automorphisms. Moreover, the uniqueness
of the decomposition of $V$ into $\langle V_1\rangle$-modules proved in \cite{Sch93} is used.

We also describe how the constructions in \cite{Hoe17,MS23} and in this text are related (see end of \autoref{sec:gdh} and \autoref{prop:leechkurz}).

\subsection*{Outline}
The paper is organised as follows: In \autoref{sec:lataut} we review lattice \voa{}s. Then we prove some results about automorphism groups of \voa{}s and apply them to lattice \voa{}s.

In \autoref{sec:orbdimrk} we recall the cyclic orbifold theory for holomorphic \voa{}s developed in \cite{EMS20a,Moe16} and state a dimension formula from \cite{EMS20b,MS23} for central charge $24$. We also prove an orbifold rank criterion.

In \autoref{sec:schellekens} we summarise what is known about the classification of \strathol{} \voa{}s $V$ of central charge $24$ and provide the context, mainly from \cite{Hoe17,MS23}, for the uniform description in this text.

In \autoref{sec:cons} we define the notion of a \good{} automorphism of a unimodular lattice \voa{} and prove that all $70$ \strathol{} \voa{}s $V$ of central charge $24$ with $V_1\neq\{0\}$ can be obtained as orbifold constructions from the $24$ Niemeier lattice \voa{}s $V_N$ associated with such automorphisms. We also study some properties of these automorphisms, classify them and show that they can be characterised as certain \gdh{}s in the sense of \cite{MS23}.

In \autoref{sec:uniqueness} we give a uniform proof of the uniqueness of a \strathol{} \voa{} $V$ of central charge $24$ with any given non-zero weight-one Lie algebra $V_1$.

\subsection*{Acknowledgements}
The authors thank Jethro van Ekeren, Shashank Kanade, Ching Hung Lam, Geoffrey Mason, Nils Scheithauer and Hiroki Shimakura for helpful discussions. They would also like to extend their gratitude to the anonymous referee. The first author was supported by the Simons Foundation (Award ID: 355294), the second author by an AMS-Simons Travel Grant.

\subsection*{Computer Calculations}
We remark that some of the computations in \autoref{sec:cons} and \autoref{sec:uniqueness} involving the isometry groups $\O(N)$ of the Niemeier lattices $N$ were performed on the computer using \texttt{Magma} \cite{Magma}.

\medskip

All \voa{}s are assumed to be complex. Lie algebras are complex and finite-dimensional.


\section{Lattice \VOA{}s and Automorphism Groups}\label{sec:lataut}
In this section we review lattice \voa{}s \cite{Bor86,FLM88,Don93} (for details see, e.g., \cite{Kac98,LL04}). Then we describe automorphism groups of \voa{}s, in particular the finite-order conjugacy classes, and apply these results to lattice \voa{}s.

\medskip

For an introduction to \voa{}s and their representation theory we refer the reader to \cite{FLM88,FHL93,LL04}.
A \voa{} $V$ is called \emph{\strat{}} if it is rational (as defined in, e.g., \cite{DLM97}), $C_2$-cofinite (or lisse), self-contragredient (or self-dual) and of CFT-type. This also implies that $V$ is simple. A \voa{} $V$ is called \emph{holomorphic} (or self-dual or meromorphic) if it is rational and the only irreducible $V$-module is $V$ itself (implying that $V$ is simple and self-contragredient). The definition of twisted modules follows the sign convention in, e.g., \cite{DLM00} as opposed to some older texts. A \vosa{} is called \emph{full} if it has the same Virasoro vector as the containing \voa{}.

In a \voa{} $V=\bigoplus_{n=0}^\infty V_n$ of CFT-type the weight-one space $V_1$ carries the structure of a (complex, finite-dimensional) Lie algebra via $[u,v]:=u_0v$ for all $u$, $v\in V_1$. If $V$ is \strat{}, then this Lie algebra is reductive, i.e.\ a direct sum of an abelian and a semisimple Lie algebra~\cite{DM04b}.

If a \voa{} $V$ is self-contragredient and of CFT-type,
then there exists a non-degenerate, invariant bilinear form $\langle\cdot,\cdot\rangle$ on $V$, which is unique up to a non-zero scalar and symmetric \cite{FHL93,Li94}. This bilinear form restricts to a non-degenerate, invariant bilinear form on the Lie algebra $V_1$. It is common to choose the normalisation $\langle\vac,\vac\rangle=-1$ where $\vac$ denotes the vacuum vector.


\subsection{Lattice \VOA{}s}\label{sec:lat}
Let $L$ be a positive-definite, even lattice, i.e.\ a free abelian group $L$ of finite rank $\rk(L)$ equipped with a positive-definite, symmetric bilinear form $\langle\cdot,\cdot\rangle\colon L\times L\to\Z$ such that $\langle\alpha,\alpha\rangle\in 2\Z$ for all $\alpha\in L$. By $\h:=L\otimes_\Z\C$ we denote the complexified lattice. The discriminant form $L'/L$ with dual lattice $L'=\{\alpha\in L\otimes_\Z\Q\mid\langle\alpha,\beta\rangle\in\Z\text{ for all }\beta\in L\}$ naturally carries the structure of a non-degenerate finite quadratic space. The lattice $L$ is called \emph{unimodular} if $L'=L$, i.e.\ if the discriminant form is trivial.

The lattice \voa{} $V_L=M_{\hat\h}(1)\otimes\C_\eps[L]$ associated with $L$ is \strat{} and of central charge $c=\rk(L)$. The definition of $V_L$ involves a choice of group $2$-cocycle $\eps\colon L\times L\to\{\pm1\}$ satisfying $\eps(\alpha,\beta)/\eps(\beta,\alpha)=(-1)^{\langle\alpha,\beta\rangle}$ for all $\alpha$, $\beta\in L$.

The irreducible $V_L$-modules $V_{\alpha+L}$, $\alpha+L\in L'/L$, are indexed by the elements of the discriminant form $L'/L$. In particular, if $L$ is unimodular, then $V_L$ is holomorphic.

\medskip

We now describe the weight-one Lie algebra $(V_L)_1$ of a lattice \voa{} $V_L$. Let $L$ be an even, positive-definite lattice. Then the set of norm-two vectors $\Phi:=\{\alpha\in L\mid\langle\alpha,\alpha\rangle=2\}$ forms a simply-laced root system. Let $R:=\spn_\Z(\Phi)\subseteq L$ denote the root sublattice generated by $\Phi$.

The weight-one Lie algebra of $V_L$ is given by
\begin{equation*}
(V_L)_1=\hh\oplus\spn_\C\left(\{1\otimes\ee_\alpha\mid\alpha\in\Phi\}\right)
\end{equation*}
where $\hh:=\{h(-1)\otimes\ee_0\mid h\in\h\}\cong\h$ is a choice of a Cartan subalgebra of $(V_L)_1$. Note that the restriction to $\hh$ of the invariant bilinear form $\langle\cdot,\cdot\rangle$ on $V_L$ normalised such that $\langle\vac,\vac\rangle=-1$ is precisely the bilinear form $\langle\cdot,\cdot\rangle$ on $L$ bilinearly extended to the complexification $\h=L\otimes_\Z\C$ under the identification of $\hh$ with $\h$.

It is easy to verify (see, e.g., Section~7.8 in \cite{Kac90}) that $(V_L)_1$ is a reductive Lie algebra of rank $\rk(L)$ with a semisimple part of rank $\rk(R)$ and an abelian part of rank $\rk(L)-\rk(R)$ and that the root system of the semisimple part of $(V_L)_1$ is exactly $\Phi$ (viewed in $\hh^*\cong\h^*$ via $\langle\cdot,\cdot\rangle$).


\subsection{Automorphism Groups}\label{sec:aut}
For any \voa{} $V$ of CFT-type $K:=\langle\{\e^{v_0}\mid v\in V_1\}\rangle$ defines a normal subgroup of $\Aut(V)$, called the \emph{inner automorphism group} (see \cite{DN99}, Section~2.3).
We call $\Aut(V)/K$ the \emph{outer automorphism group} of $V$.

Since \voa{} automorphisms are grading-preserving, there is a restriction homomorphism $r\colon\Aut(V)\to\Aut(V_1)$, i.e.\ automorphisms of the \voa{} $V$ restrict to Lie algebra automorphisms of $V_1$. This homomorphism $r$ is in general not surjective. However, it follows from the definition of $K$ and the relation $r(\e^{v_0})=\e^{\ad_v}$ that $r(K)=\Inn(V_1)$, i.e.\ every inner automorphism of $V_1$ can be extended to an inner automorphism of $V$.

Since $r(K)\subseteq\Inn(V_1)$, $r$ induces a homomorphism $\tilde{r}\colon\Aut(V)/K\to\Out(V_1)$.
Moreover, $\tilde{r}$ is injective if and only if
\begin{equation*}\label{eq:ass1}
\ker(r)\subseteq K.\tag{A}
\end{equation*}
In the following we assume that \eqref{eq:ass1} holds, or equivalently that $r^{-1}(\Inn(V_1))=K$.
This excludes, in particular, the case of a \voa{} with $V_1=\{0\}$ and $\Aut(V)\neq\{\id\}$, like, for example, the Moonshine module $V^\natural$. Then we have the following commutative diagram with exact columns:
\begin{equation*}
\begin{tikzcd}
K\arrow[two heads]{rr}{\phantom{|_K}r|_K}\arrow[hookrightarrow]{d}&&\Inn(V_1)\arrow[hookrightarrow]{d}\\
\Aut(V)\arrow{rr}{r}\arrow[two heads]{d}&&\Aut(V_1)\arrow[two heads]{d}\\
\Aut(V)/K\arrow[hookrightarrow]{rr}{\tilde{r}}&&\Out(V_1)
\end{tikzcd}
\end{equation*}

Since $\tilde{r}$ is injective, the outer automorphism group $\Aut(V)/K$ is finite whenever $\Out(V_1)$ is, for example, if $V_1$ is semisimple.
We shall see that $\Aut(V)/K$ is finite if $V$ is a lattice \voa{}, even when $V_1$ is not semisimple. One might speculate that $\Aut(V)/K$ is finite for any sufficiently regular (e.g., \strat{}) \voa{} $V$.

Note that $\ker(r)$, the subgroup of $\Aut(V)$ acting trivially on $V_1$, was introduced as \emph{inertia group} $I(V)$ in \cite{LS20b} and studied for a few examples of \strathol{} \voa{}s of central charge $24$ (see also \autoref{rem:kernel}).

\medskip

First, we describe finite-order, inner automorphisms of \voa{}s up to conjugacy. The following result is immediate:
\begin{prop}\label{prop:equiv1}
Let $V$ be a \voa{} of CFT-type and $g\in\Aut(V)$, and assume that \eqref{eq:ass1} holds. Then the following are equivalent:
\begin{enumerate}
\item The automorphism $g$ is inner, i.e.\ $g\in K$.
\item The restriction $r(g)$ is inner, i.e.\ $r(g)\in\Inn(V_1)$.
\end{enumerate}
\end{prop}
%

From now on, we assume that $V_1$ is reductive, i.e.\ a direct sum of an abelian and a semisimple part.
If necessary, we fix a choice $\hh$ of Cartan subalgebra of $V_1$.

We then consider the subgroup $T:=\langle\{\e^{v_0}\mid v\in\hh\}\rangle=\{\e^{v_0}\mid v\in\hh\}$ of $K$, which is abelian since $\hh$ is and because $[u,v]=0$ implies that $[u_0,v_0]=0$ for $u,v\in V_1$ by Borcherds' identity. We shall assume in the following, strengthening \eqref{eq:ass1}, that
\begin{equation*}\label{eq:ass2}
\ker(r)\subseteq T\tag{B}
\end{equation*}
for the choice $\hh$ of Cartan subalgebra of $V_1$. But then \eqref{eq:ass2} already holds for all Cartan subalgebras of $V_1$. Conditions \eqref{eq:ass1} and \eqref{eq:ass2} are satisfied, for example, for any lattice \voa{} by Lemma~2.5 in \cite{DN99}.

\begin{prop}\label{prop:equiv2}
In the situation of \autoref{prop:equiv1}, assume that $V_1$ is reductive, that \eqref{eq:ass2} holds and that $g\in\Aut(V)$ has finite order.\footnote{The finite-order assumption is missing in the published version of this text.} Then the items in \autoref{prop:equiv1} are equivalent to:
\begin{enumerate}
\item[(3)] The automorphism $g$ is conjugate in $K$ to $\e^{v_0}\in T$ for some $v\in\hh$.
\end{enumerate}
\end{prop}
\begin{proof} That (3) implies (1) is clear.

To see that (2) implies (3) suppose $r(g)\in\Inn(V_1)$. Then by (the extension to reductive Lie algebras of) Proposition~8.1 in \cite{Kac90} $r(g)$ is conjugate to $\e^{\ad_v}$ for some $v\in\hh$. In fact, since any two Cartan subalgebras are conjugate by an inner automorphism, it follows from the proof of this result that $r(g)$ and $\e^{\ad_v}$ are conjugate under some element in $\Inn(V_1)$, which lifts to some $k\in K$. Then $r(g)=r(ke^{v_0}k^{-1})$ or $r(\e^{v_0}k^{-1}g^{-1}k)=\id$. By condition~\eqref{eq:ass2} this means that $\e^{v_0}k^{-1}g^{-1}k=\e^{w_0}$ for some $w\in\hh$, which we can rewrite as $g=k\e^{(v-w)_0}k^{-1}$, using the commutativity of $v_0$ and $w_0$.
\end{proof}

\begin{rem}\label{rem:kernel}
Evidently, if $V_1$ is semisimple, then
\begin{equation*}
T\cap\ker(r)=\{\e^{(2\pi\i)v_0}\mid v\in\hh,\,\e^{(2\pi\i)\ad_v}=\id_{V_1}\}=\{\e^{(2\pi\i)v_0}\mid v\in P^\vee\}
\end{equation*}
where $P^\vee\subseteq\hh$ denotes the coweight lattice of $V_1$. Then condition~\eqref{eq:ass2} is equivalent to $\ker(r)=\{\e^{(2\pi\i)v_0}\mid v\in P^\vee\}$. The latter is shown to hold for some examples of \strathol{} \voa{}s of central charge $24$ in \cite{LS20b}, Remark~6.6. In fact, Ching Hung Lam informed us that he can prove the statement for all \strathol{} \voa{}s of central charge $24$ with non-zero (and semisimple) weight-one Lie algebra. The proof uses the result from \cite{Hoe17} that these \voa{}s are simple-current extensions of a certain dual pair (see \autoref{thm:hoehn}).
\end{rem}
We remark that condition~\eqref{eq:ass1} implies condition~\eqref{eq:ass2} if one assumes that $K$ is a complex Lie group.

We also note:
\begin{prop}\label{prop:equiv3}
In the situation of \autoref{prop:equiv1}, assume that $V_1$ is reductive and that $g\in\Aut(V)$ has finite order. Then also $V_1^g$ is reductive and the items in \autoref{prop:equiv1} are equivalent to:
\begin{enumerate}
\item[(4)] $\rk(V_1^g)=\rk(V_1)$.
\end{enumerate}
\end{prop}
\begin{proof}
The proof follows from Proposition~8.6 in \cite{Kac90}.
\end{proof}

\medskip

In the following we characterise finite-order automorphisms of $V$ up to conjugacy. For ease of presentation we shall assume that $V_1$ is semisimple. It is however not difficult to extend the results to abelian or reductive Lie algebras.

Note that there is a non-degenerate, symmetric, invariant bilinear form $(\cdot,\cdot)$ on $V_1$, which restricts to a
non-zero multiple of the Killing form on each simple ideal of $V_1$. Upon fixing a choice of this form we may identify the Cartan subalgebra $\hh$ with its dual $\hh^*$.
In particular, we may view the roots of $V_1$ as elements of $\hh$. Usually, $(\cdot,\cdot)$ is normalised such that the long roots have norm $2$. If $V$ is self-contragredient,
then $(\cdot,\cdot)$ and the invariant bilinear form $\langle\cdot,\cdot\rangle$ on $V$ agree up to multiplication by a non-zero scalar on each simple ideal of $V_1$.

In addition to fixing a Cartan subalgebra $\hh$ of $V_1$, for the following results let us also fix a choice of simple roots $\Delta$. Then, given an outer (or diagram) automorphism $\mu_0\in\Out(V_1)=\Aut(V_1)/\Inn(V_1)$, there is a \emph{standard lift} $\mu\in\Aut(V_1)$ of $\mu_0$ (defined in (7.9.2) and (7.10.1) in \cite{Kac90} and called diagram automorphism there) such that $\mu_0=\mu\Inn(V_1)$ and $\mu$ fixes the Cartan subalgebra $\hh$ and the simple roots $\Delta$ setwise.

Moreover, for all outer automorphisms $\mu_0\in\tilde{r}(\Aut(V)/K)\subseteq\Out(V_1)$ and their standard lifts $\mu\in\Aut(V_1)$ we also fix \emph{choices of extensions} of $\mu$ to $\tilde\mu\in\Aut(V)$. By the above commutative diagram, such an extension from $\Aut(V_1)$ to $\Aut(V)$ exists precisely for those automorphisms in $\Aut(V_1)$ that project to $\tilde{r}(\Aut(V)/K)\subseteq\Out(V_1)$.

The next lemma shows that the automorphisms in $\Aut(V)$ of finite order can be conjugated by elements in $K$ into $\Aut(V)_{\{\hh\}}=\{g\in\Aut(V)\mid g(\hh)\subseteq\hh\}$, the setwise stabiliser of $\hh$ in $\Aut(V)$. (Note that $\Aut(V)_{\{\hh\}}=r^{-1}(\Aut(V_1)_{\{\hh\}})$.) In fact, they can be conjugated into $\Aut(V)_{\{\Delta\}}$, the automorphisms fixing the simple roots $\Delta$ setwise.
\begin{lem}\label{lem:auts}
Let $V$ be a \voa{} of CFT-type such that $V_1$ is semisimple and \eqref{eq:ass2} holds. Let $g\in\Aut(V)$ be of finite order. Then $g$ is conjugate under an automorphism in $K$ to an automorphism of the form
\begin{equation*}
\tilde{\mu}\,\e^{(2\pi\i)v_0}
\end{equation*}
for some $v\in\hh^\mu$ where $\mu\in\Aut(V_1)$ is the standard lift of an outer automorphism in $\tilde{r}(\Aut(V)/K)\subseteq\Out(V_1)$ and $\tilde{\mu}$ the choice of extension to $\Aut(V)$.
\end{lem}
Note that $\tilde{\mu}$ and $\e^{(2\pi\i)v_0}$ commute since $v\in\hh^\mu$, i.e.\ $\tilde\mu v=v$. More generally, $g\e^{v_0}=\e^{(gv)_0}g$ for any $g\in\Aut(V)$ and $v\in V_1$.
\begin{proof}
Consider $r(g)\in\Aut(V_1)$. Again, by (the proof of) Proposition~8.1 in \cite{Kac90} $r(g)$ is conjugate under an inner automorphism, say $r(k)$ for some $k\in K$, to $\mu\e^{(2\pi\i)\ad_v}$ for some $v\in\hh$ where $\mu$ is the standard lift of an outer automorphism $\mu_0$. This outer automorphism $\mu_0$ is in $\tilde{r}(\Aut(V)/K)$ since $r(g)$ projects to an element in $\tilde{r}(\Aut(V)/K)\subseteq\Out(V_1)$.
Hence $r(g)=r(k)\mu\e^{{(2\pi\i)\ad_v}}r(k)^{-1}=r(k\tilde{\mu}\e^{(2\pi\i)v_0}k^{-1})$ where $\tilde\mu$ is the choice of extension of $\mu$ to $\Aut(V)$.
Then $r(k^{-1}g^{-1}k\tilde{\mu}\e^{(2\pi\i)v_0})=\id$, i.e.\ $k^{-1}g^{-1}k\tilde{\mu}\e^{(2\pi\i)v_0}=\e^{(2\pi\i)w_0}$ for some $w\in\hh$ by condition~\eqref{eq:ass2}. This proves that $g$ is conjugate under $k$ to $\tilde{\mu}\e^{(2\pi\i)(v-w)_0}$. Finally, by Lemma~8.3 in \cite{EMS20b} we may assume that $v-w\in\hh^\mu$.
\end{proof}

\begin{lem}
In \autoref{lem:auts} it suffices, up to conjugacy in $\Aut(V)$,
to let $\mu$ be from a fixed set of (standard lifts of) representatives of the conjugacy classes of $\tilde{r}(\Aut(V)/K)\cong\Aut(V)/K$.
\end{lem}
\begin{proof}
Consider the automorphism $\tilde{\mu}\e^{(2\pi\i)v_0}$ from the above result. It is clear that $\tilde{r}(\tilde{\mu}\e^{(2\pi\i)v_0}K)=\mu\Inn(V_1)$. Now suppose $\mu\Inn(V_1)$ is conjugate to some $\nu\Inn(V_1)$ in $\tilde{r}(\Aut(V)/K)\subseteq\Out(V_1)$ where $\nu\in\Aut(V_1)$ is again the standard lift of the outer automorphism $\nu\Inn(V_1)$. Then there is a $\tau\Inn(V_1)\in\tilde{r}(\Aut(V)/K)\subseteq\Out(V_1)$ such that $\tau\mu\tau^{-1}\Inn(V_1)=\nu\Inn(V_1)$. Consider the choices of extensions $\tilde{\nu},\tilde{\tau}\in\Aut(V)$ with $r(\tilde{\nu})=\nu$ and $r(\tilde{\tau})=\tau$. Then
\begin{align*}
\tilde{r}(\tilde{\tau}\tilde{\mu}\e^{(2\pi\i)v_0}\tilde{\tau}^{-1}K)&=\tilde{r}(\tilde{\tau}K)\tilde{r}(\tilde{\mu}\e^{(2\pi\i)v_0}K)\tilde{r}(\tilde{\tau}^{-1}K)=\tau\mu\tau^{-1}\Inn(V_1)\\
&=\nu\Inn(V_1)=\tilde{r}(\tilde{\nu}K)
\end{align*}
and hence $\tilde{\tau}\tilde{\mu}\e^{(2\pi\i)v_0}\tilde{\tau}^{-1}K=\tilde{\nu}K$. This implies that $\tilde{\tau}\tilde{\mu}\e^{(2\pi\i)v_0}\tilde{\tau}^{-1}=\tilde{\nu}k$, i.e.\ $\tilde{\mu}\e^{(2\pi\i)v_0}$ is conjugate to $\tilde{\nu}k$ for some $k\in K$. By the above lemma, this has to be conjugate under an automorphism in $K$ to some $\tilde{\sigma}\e^{(2\pi\i)w_0}$. But because it is conjugate under an automorphism in $K$, we know that $\tilde{\nu}K=\tilde{\nu}kK=\tilde{\sigma}\e^{(2\pi\i)w_0}K=\tilde{\sigma}K$, which implies that $\nu\Inn(V_1)=\tilde{r}(\tilde{\nu}K)=\tilde{r}(\tilde{\sigma}K)=\sigma\Inn(V_1)$. But then $\nu=\sigma$ (and $\tilde{\nu}=\tilde{\sigma}$) because these were some fixed standard lifts (and extensions).

In total, we have shown that $\tilde{\mu}\e^{(2\pi\i)v_0}$ is conjugate to $\tilde{\nu}\e^{(2\pi\i)w_0}$ for some $w\in\hh^\nu$ if $\mu\Inn(V_1)$ is conjugate to $\nu\Inn(V_1)$. This proves the claim.
\end{proof}

For $g\in\Aut(V)_{\{\hh\}}$ (and assuming that $r(g)$ has order $n$) we define the projection map $\pi_g\colon\hh\to\hh^g$ by $v\mapsto\frac{1}{n}\sum_{i=0}^{n-1}g^iv$ for $v\in\hh$. Then $\hh^g=\pi_g(\hh)$. We also define the subgroup
\begin{equation*}
P:=\{v\in\hh\mid\e^{(2\pi\i)v_0}=\id_V\}
\end{equation*}
of the Cartan subalgebra $\hh$ acting trivially on $V$ via $\e^{(2\pi\i)(\cdot)}$.

\begin{lem}
In \autoref{lem:auts} it suffices, up to conjugacy in $\Aut(V)$, to select $v$ from a fixed set of representatives for $\hh^\mu/\pi_\mu(P)$.
\end{lem}
\begin{proof}
Let $w\in P$. Then $\tilde{\mu}\e^{(2\pi\i)v_0}=\tilde{\mu}\e^{(2\pi\i)v_0}\e^{(2\pi\i)w_0}=\tilde{\mu}\e^{(2\pi\i)(v+w)_0}$ is conjugate to $\tilde{\mu}\e^{(2\pi\i)(v+\pi_\mu(w))_0}$ by Lemma~8.3 in \cite{EMS20b}.
\end{proof}

In summary we find:
\begin{prop}[Finite-Order Conjugacy Classes]\label{prop:autconj}
Let $V$ be a \voa{} of CFT-type such that $V_1$ is semisimple and \eqref{eq:ass2} holds. Let $g\in\Aut(V)$ be of finite order. Then $g$ is conjugate to an automorphism of the form
\begin{equation*}
\tilde{\mu}\,\e^{(2\pi\i)v_0}
\end{equation*}
where $\mu\in\Aut(V_1)$ is from a fixed set of (standard lifts of) representatives of the conjugacy classes of $\tilde{r}(\Aut(V)/K)\subseteq\Out(V_1)$ and $\tilde{\mu}$ the choice of extension to $\Aut(V)$, and $v$ is from a fixed set of representatives of $\hh^\mu/\pi_\mu(P)$.
\end{prop}
Two conjugate automorphisms in $\Aut(V)$ must project to the same conjugacy class in $\Aut(V)/K$. However, two automorphisms $\tilde{\mu}\e^{(2\pi\i)v_0}$ and $\tilde{\mu}\e^{(2\pi\i)v'_0}$ where $v$ and $v'$ represent different classes in $\hh^\mu/\pi_\mu(P)$ may still be conjugate.

In the lattice case we shall refine this result by further considering orbits under the action of $\Aut(V)_{\{\hh\}}/\Aut(V)_\hh$, the quotient of the setwise stabiliser by the pointwise stabiliser, on $\hh$.


\subsection{Automorphism Groups of Lattice \VOA{}s}\label{sec:autlat}
In the following we specialise the previous discussion to a lattice \voa{} $V_L=M_{\hat\h}(1)\otimes\C_\eps[L]$ for an even, positive-definite lattice $L$.

The group of isometries (or automorphisms) of $L$ is denoted by $\O(L)$.
An automorphism $\nu\in\O(L)$ and a function $\eta\colon L\to\{\pm1\}$ satisfying $\eta(\alpha)\eta(\beta)/\eta(\alpha+\beta)=\eps(\alpha,\beta)/\eps(\nu\alpha,\nu\beta)$ for all $\alpha,\beta\in L$ define a \emph{lift} $\phi_\eta(\nu)$ of $\nu$, acting on $\C_\eps[L]=\bigoplus_{\alpha\in L}\C\ee_\alpha$ as $\phi_\eta(\nu)(\ee_\alpha)=\eta(\alpha)\ee_{\nu\alpha}$ for $\alpha\in L$ and on $M_{\hat\h}(1)$ as $\nu$ in the obvious way, and the automorphisms obtained in this way form the subgroup $\O(\smash{\hat{L}})\subseteq\Aut(V_L)$ (see, e.g., \cite{FLM88,Bor92}). The following sequence is exact:
\begin{equation*}
1\longrightarrow\Hom(L,\{\pm1\})\longrightarrow\O(\hat{L})\stackrel{\text{--}}{\longrightarrow}\O(L)\longrightarrow 1.
\end{equation*}
The injection is given by $\lambda\mapsto\phi_\lambda(\id)$ and the surjection by $\phi_\eta(\nu)\mapsto\overline{\phi_\eta(\nu)}=\nu$.

A lift $\phi_\eta(\nu)\in\O(\hat{L})$ is called \emph{standard lift} if the restriction of $\eta$ to the fixed-point sublattice $L^\nu\subseteq L$ is trivial. Standard lifts always exist \cite{Lep85}, and all standard lifts of a given lattice automorphism $\nu$ are conjugate in $\Aut(V_L)$ \cite{EMS20a}. Standard lifts appear in the definition of twisted modules of lattice \voa{}s \cite{DL96,BK04}. For convenience, we shall fix a choice $\hat{\nu}$ of standard lift for all $\nu\in\O(L)$.

If $\nu\in\O(L)$ has order $m$ and $\hat\nu$ is a standard lift of $\nu$, then $\hat\nu$ has order $m$ if $m$ is odd or if $m$ is even and $\langle\alpha,\nu^{m/2}\alpha\rangle\in 2\Z$ for all $\alpha\in L$, and $\hat\nu$ has order $2m$ otherwise. In the latter case we say that $\nu$ exhibits \emph{order doubling}.

When expressing powers of automorphisms $\hat{\nu}$ in $\Aut(V_L)$ a small complication arises due to the fact that if $\hat{\nu}$ is a standard lift of $\nu\in\O(L)$, $\hat{\nu}^i$ is not necessarily a standard lift of $\nu^i$. However, there is a vector $s_i$ in $(1/2)(L')^{\nu^i}=(1/2)(\pi_{\nu^i}(L))'$ such that
\begin{equation*}
\hat{\nu}^i=\widehat{(\nu^i)}\e^{-(2\pi\i)s_i(0)}
\end{equation*}
where $\widehat{(\nu^i)}$ is a standard lift of $\nu^i$. The minus sign is a convention related to the sign convention in the definition of twisted modules.

Note that $s_m\in(1/2)L'$ (with $\nu$ of order $m$) can be taken to be zero if and only if $\nu$ does not exhibit order doubling.

\medskip

We review some well-known facts about the isometry group $\O(L)$. Recall that $\Phi$ denotes the simply-laced root system comprised of the norm-two vectors in $L$. The Weyl group $W\subseteq\O(L)$ is defined as the normal subgroup generated by the reflections about the hyperplanes orthogonal to the roots in $\Phi$.

The automorphism group $\O(L)$ of $L$ is a split extension $W{:}H$, i.e.\ $H:=\O(L)/W$ and the short exact sequence
\begin{equation*}
1\longrightarrow W\longrightarrow\O(L)\longrightarrow H\longrightarrow 1
\end{equation*}
is right split (see, e.g., Section~1 of \cite{Bor87}). In other words, $H$ is isomorphic to a subgroup of $\O(L)$ and, in fact, given a choice of simple roots $\Delta\subseteq\Phi$, the setwise stabiliser of $\Delta$ in $\O(L)$
\begin{equation*}
H_\Delta:=\O(L)_{\{\Delta\}}
\end{equation*}
is isomorphic to $H$.

\medskip

Recall that for any \voa{} $V$ of CFT-type the normal subgroup $K=\langle\{\e^{v_0}\mid v\in V_1\}\rangle$ of $\Aut(V)$ is the inner automorphism group of $V$. For a lattice \voa{} $V_L$ it was shown in \cite{DN99}, Theorem~2.1, that
\begin{equation*}
\Aut(V_L)=\O(\hat{L})K
\end{equation*}
and that the outer automorphism group $\Aut(V_L)/K$ is isomorphic to a quotient group of $\O(L)$. More precisely, $\Aut(V_L)/K$ is isomorphic to $\O(\hat{L})/(K\cap\O(\hat{L}))$, which, since both $\O(\hat{L})$ and $K$ contain the subgroup $\Hom(L,\{\pm1\})$, is in turn isomorphic to $\O(L)/\overline{K\cap\O(\hat{L})}$, i.e.\
\begin{equation*}
\Aut(V_L)/K\cong\O(\hat{L})/(K\cap\O(\hat{L}))\cong\O(L)/\overline{K\cap\O(\hat{L})},
\end{equation*}
where $\overline{K\cap\O(\hat{L})}$ denotes the image of $K\cap\O(\hat{L})$ under the projection from $\O(\hat{L})$ to $\O(L)$. In particular, $\Aut(V_L)/K$ is finite since $\O(L)$ is finite for a positive-definite lattice $L$. We shall further simplify $\Aut(V_L)/K$ below.

\medskip

In the following, we choose the Cartan subalgebra $\hh=\{h(-1)\otimes\ee_0\mid h\in\h\}\cong\h$ of the reductive (simply-laced) Lie algebra $(V_L)_1$ (see \autoref{sec:lat}) if necessary. With this choice, $\Phi$ is the root system of $(V_L)_1$ and $W$ the Weyl group of $(V_L)_1$.

Recall that the automorphisms $T=\{\e^{v_0}\mid v\in\hh\}$ define an abelian subgroup of~$K$. It is easy to describe the action of $T\subseteq\Aut(V_L)$ on a lattice \voa{} $V_L=\bigoplus_{\alpha\in L}M_{\hat\h}(1)\otimes\ee_\alpha$, which is naturally graded by $L$. Indeed, for $v=h(-1)\otimes\ee_0\in\hh$, $h\in\h$, the automorphism $\e^{v_0}=\e^{h(0)}$ acts by multiplication with $\e^{\langle h,\alpha\rangle}$ on the graded component $M_{\hat\h}(1)\otimes\ee_\alpha$, $\alpha\in L$.

Hence, the group $P=\{v\in\hh \mid \e^{(2\pi\i)v_0}=\id_V\}$ defined above is given by
\begin{equation*}
P=\{h(-1)\otimes\ee_0 \mid h\in L'\}\cong L',
\end{equation*}
with the dual lattice $L'$ of $L$. Also note that
\begin{equation*}
\ker(\overline{\phantom{o}})=\O(\hat{L})\cap T=\{\e^{(2\pi\i)h(0)}\mid h\in L'/2\}\cong\Hom(L,\{\pm1\})
\end{equation*}
where $\overline{\phantom{o}}$ denotes the projection map $\O(\hat{L})\to\O(L)$.

\begin{rem}\label{rem:k0}
While the projection $\Aut(V_L)\to\Aut(V_L)/K\cong\O(L)/\overline{K\cap\O(\hat{L})}$ is independent of the choice of a Cartan subalgebra of $(V_L)_1$, a projection to $\O(L)$ can in general only be defined involving such a choice. Indeed, choosing the Cartan subalgebra $\hh=\{h(-1)\otimes\ee_0\mid h\in\h\}$, the setwise stabiliser of $\hh$ is given by
\begin{equation*}
\Aut(V_L)_{\{\hh\}}=\O(\hat{L})T
\end{equation*}
while the pointwise stabiliser of $\hh$ is
\begin{equation*}
\Aut(V_L)_\hh=T.
\end{equation*}
Then the quotient $\Aut(V_L)_{\{\hh\}}/\Aut(V_L)_\hh$, which acts (faithfully) on $\hh$, becomes
\begin{equation*}
\Aut(V_L)_{\{\hh\}}/\Aut(V_L)_\hh=\O(\hat{L})T/T\cong\O(\hat{L})/(\O(\hat{L})\cap T)\cong\O(L)
\end{equation*}
since $\O(\hat{L})\cap T\cong\Hom(L,\{\pm1\})$ are exactly the lifts of $\id\in\O(L)$. Explicitly, $\phi_\eta(\nu)\e^{(2\pi\i)h(0)}$ projects to $\nu$ for any lift $\phi_\eta(\nu)$ of $\nu$ and any $\e^{(2\pi\i)h(0)}\in T$.
\end{rem}

We return to the quotient $\Aut(V_L)/K$:
\begin{prop}[Outer Automorphisms]
Let $L$ be an even, positive-definite lattice and $V_L$ the associated lattice \voa{}. Then $\overline{K\cap\O(\hat{L})}=W$ so that the outer automorphism group satisfies
\begin{equation*}
\Aut(V_L)/K\cong\O(L)/W=H
\end{equation*}
where $W$ is the Weyl group of the root system $\Phi$, i.e.\ the group generated by the reflections about the norm-two vectors in $L$, which is also the Weyl group of the reductive (and simply-laced) Lie algebra $(V_L)_1$.
\end{prop}
\begin{proof}
Like the Weyl group $W$, the inner automorphisms in $\Inn((V_L)_1)_{\{\hh\}}$ act by definition on the Cartan subalgebra $\hh$. In fact, it is well known that for a simple Lie algebra $\g$ the image of the restriction map $\Inn(\g)_{\{\hh\}}\to\Aut(\hh)$ is the Weyl group $W$, i.e.\ the automorphisms of the Cartan subalgebra that are in the Weyl group are exactly those that can be extended to inner automorphisms of the whole Lie algebra $\g$. This result remains true if $\g$ is reductive. Equivalently, since $\Inn(\g)_\hh$ is the kernel of this restriction map, $W\cong\Inn(\g)_{\{\hh\}}/\Inn(\g)_\hh$.

Given the \voa{} $V_L$ with reductive Lie algebra $(V_L)_1$, since inner automorphisms of $(V_L)_1$ are exactly the restrictions of the inner automorphisms of $V_L$, it follows that
\begin{equation*}
W\cong\Inn((V_L)_1)_{\{\hh\}}/\Inn((V_L)_1)_\hh\cong K_{\{\hh\}}/K_\hh.
\end{equation*}
This implies that $W\subseteq\overline{K\cap\O(\hat{L})}$. On the other hand, by intersecting $K$ with the expressions in \autoref{rem:k0} we obtain
\begin{align*}
W&\cong K_{\{\hh\}}/K_\hh=(K\cap\O(\hat L)T)/T=T(K\cap\O(\hat L))/T\\
&\cong(K\cap\O(\hat L))/(T\cap\O(\hat L))=\overline{K\cap\O(\hat L)}.
\end{align*}
Hence, $W=\overline{K\cap\O(\hat{L})}$.
\end{proof}

Recall that it is shown in Lemma~2.5 of \cite{DN99} that $\ker(r)\subseteq T$ for any lattice \voa{} $V_L$, i.e.\ conditions \eqref{eq:ass1} and \eqref{eq:ass2} are satisfied. Then $\ker(r)$ is given by the expression in \autoref{rem:kernel}:
\begin{rem}
Let $L$ be an even, positive-definite lattice and $V_L$ the associated lattice \voa{}. Assume for simplicity that $(V_L)_1$ is semisimple. Then
\begin{align*}
\ker(r)&=\{\e^{(2\pi\i)v_0}\mid v=h(-1)\otimes\ee_0\text{ for }h\in R'\}=\{\e^{(2\pi\i)h(0)}\mid h\in R'\}\\
&\cong R'/L'
\end{align*}
where $R'$ is the dual lattice of the lattice $R$ generated by the root system $\Phi$.
\end{rem}

Since condition~\eqref{eq:ass2} is satisfied, we can give the following characterisation of inner automorphisms:
\begin{prop}
Let $L$ be an even, positive-definite lattice, $V_L$ the associated lattice \voa{} and $g\in\Aut(V_L)$ of finite order. Then the following are equivalent:
\begin{enumerate}
\item The automorphism $g$ is inner, i.e.\ $g\in K$.
\item The restriction $r(g)$ is an inner automorphism of the Lie algebra $(V_L)_1$.
\item The rank of the fixed-point Lie subalgebra $(V_L^g)_1$ equals $\rk((V_L)_1)=\rk(L)$.
\item The automorphism $g$ is conjugate in $\Aut(V_L)$ (even in $K$) to $\e^{v_0}$ for some $v\in\hh$. More precisely, since $g$ has finite order, $v=(2\pi\i)h(-1)\otimes\ee_0$ for some $h\in L\otimes_\Z\Q$.
\item The \fpvosa{} $V_L^g$ is isomorphic to a lattice \voa{} $V_K$ for some sublattice $K$ of $L$ of full rank, i.e., $V_K$ is a full vertex operator subalgebra of $V_L$.
\end{enumerate}
\end{prop}
\begin{proof}
The equivalence of $(1)$ to $(4)$ is immediate with \autoref{prop:equiv1}, \autoref{prop:equiv2} and \autoref{prop:equiv3}, noting that $\rk((V_L)_1)=\rk(L)$.

For (5) implies (3): The \fpvosa{} $V_L^g$ and $V_L$ have the same central charge $c$, which equals $\rk(L)$ but also $\rk(K)$ if $V_L^g\cong V_K$. As explained above, the Lie rank of $(V_K)_1\cong (V_L^g)_1$ equals $\rk(K)$, which gives $\rk((V_L^g)_1)=\rk((V_K)_1)=\rk(K)=c=\rk(L)=\rk((V_L)_1)$.

For (4) implies (5): The \fpvosa{} under $\e^{v_0}=\e^{(2\pi\i)h(0)}$ for some $h\in L\otimes_\Z\Q$ is given by the lattice \voa{} $V_K$ with $K=\{\alpha\in L\mid\langle\alpha,h\rangle\in\Z\}$. Since $g$ and $\e^{v_0}$ are conjugate, they have isomorphic \fpvosa{}s, which implies $V_L^g\cong V_K$.
\end{proof}

\medskip

Finally, we state the main result of this section, a description of the automorphisms of $V_L$ up to conjugacy. Given an even, positive-definite lattice $L$ and an automorphism $\nu\in\O(L)$ (of order $m$) we may consider the projection $\pi_\nu\colon\h\to\h^\nu$, $\pi_\nu=\frac{1}{m}\sum_{i=0}^{m-1}\nu^i$, of $\h=L\otimes_\Z\C$ onto $\h^\nu$, the elements of $\h$ fixed by $\nu$. Also recall that $\hat\nu\in\O(\hat{L})$ denotes a fixed choice of standard lift for every $\nu\in\O(L)$.

With the above preparations, it is easy to prove the following specialisation of \autoref{prop:autconj}. We may drop the assumption that the weight-one Lie algebra $(V_L)_1$ is semisimple, which was merely made for convenience.
\begin{thm}[Finite-Order Conjugacy Classes]\label{thm:latconj}
Let $L$ be an even, positive-definite lattice and $V_L$ the corresponding \voa{}. Let $g\in\Aut(V_L)$ of finite order. Then $g$ is conjugate to an automorphism of the form
\begin{equation*}
\hat{\nu}\,\e^{(2\pi\i)h(0)}
\end{equation*}
where $\nu$ is from a fixed set of representatives of the conjugacy classes of $H_\Delta\subseteq\O(L)$ and $\hat\nu$ the choice of standard lift in $\O(\hat{L})\subseteq\Aut(V_L)$, and $h$ is from a fixed set of orbit representatives of the action of the centraliser $C_{\O(L)}(\nu)$ on $\h^{\nu}/\pi_\nu(L')$.
\end{thm}
Note that $\pi_\nu(L')=(L^\nu)'$. The automorphisms $\hat{\nu}$ and $\e^{(2\pi\i)h(0)}$ commute since $h\in\h^\nu$. Moreover,
since $g$ has finite order, $h$ is actually in $L^\nu\otimes_\Z\Q=\pi_\nu(L\otimes_\Z\Q)$ rather than $\h^\nu=L^\nu\otimes\C$.
\begin{proof}
If $(V_L)_1$ is abelian (i.e.\ if $L$ has no vectors of norm $2$), then the Weyl group is trivial and the assertion was already proved in \cite{MS23}. To be precise, it is only stated there for the Leech lattice $\Lambda$, but the proof carries over almost verbatim to any positive-definite, even lattice $L$ without norm-$2$ vectors, the only difference being that $L$ in contrast to $\Lambda$ does not have to be unimodular.

In the following, let us assume that $(V_L)_1$ is semisimple. The general case when $(V_L)_1$ is reductive, i.e.\ a direct sum of an abelian and a semisimple Lie algebra, follows from the abelian and semisimple cases.

We apply \autoref{prop:autconj}. Let $\mu_0$ be a representative of a conjugacy class of $H\cong\tilde{r}(\Aut(V_L)/K)\subseteq\Out((V_L)_1)$. Then the standard lift $\mu\in\Aut((V_L)_1)$ of $\mu_0$ as described
in \cite{Kac90}, preserving the choice of Cartan subalgebra $\hh$ and the choice of simple roots $\Delta$ is exactly the restriction to $(V_L)_1$ of $\hat\nu\in\O(\hat{L})$ times some element in $T$ where $\nu$ is the element in $H_\Delta\subseteq\O(L)$ corresponding to $\mu_0$. Hence, modulo $T$, we may take $\tilde{\mu}=\hat{\nu}$.

Indeed, both $\tilde{\mu}$ and $\hat{\nu}$ fix the Cartan subalgebra $\hh$ setwise, i.e.\ they are in $\Aut(V_L)_{\{\hh\}}=\O(\hat{L})T$ (see \autoref{rem:k0}). Both their projections modulo $T$ to $\O(L)$ lie in the subgroup $H_\Delta\cong H$, which is clear for $\hat{\nu}$ by definition and follows for $\tilde{\mu}$ since it fixes $\Delta$ setwise. But then these projections coincide and $\tilde{\mu}T=\hat{\nu}T$.

Recall that $P\cong L'$. Then, by \autoref{prop:autconj}, $g$ is conjugate to $\hat{\nu}\,\e^{(2\pi\i)h(0)}$ where $\nu$ is from a set of representatives of the conjugacy classes of $H_\Delta\subseteq\O(L)$ and $h$ is from a set of representatives of $\h^{\nu}/\pi_\nu(L')$.

\smallskip

It remains to show that it suffices to let $h$ be from a set of orbit representatives of the action of $C_{\O(L)}(\nu)$ on $\h^\nu/\pi_\nu(L')$. To this end, let $h$ and $h'$ in $\h^\nu$ such that $h'+\pi_\nu(L')=\tau h+\pi_\nu(L')$ for some $\tau\in C_{\O(L)}(\nu)$. We want to show that $\hat{\nu}\,\e^{(2\pi\i)h(0)}$ and $\hat{\nu}\,\e^{(2\pi\i)h'(0)}$ are conjugate.

The automorphism $(\hat{\nu}\,\e^{(2\pi\i)h'(0)})^{-1}\hat{\tau}\hat{\nu}\,\e^{(2\pi\i)h(0)}\hat{\tau}^{-1}$ is in $\Aut(V_L)_\hh=T$, i.e.\ equal to $\e^{(2\pi\i)f(0)}$ for some $f\in\h$. Hence, $\hat{\nu}\,\e^{(2\pi\i)h(0)}$ is conjugate to $\hat{\nu}\,\e^{(2\pi\i)(h'+f)(0)}$, which is conjugate to $\hat{\nu}\,\e^{(2\pi\i)(h'+\pi_\nu(f))(0)}$.

On the other hand, $(\hat{\nu}\,\e^{(2\pi\i)h'(0)})^{-1}\hat{\tau}\hat{\nu}\,\e^{(2\pi\i)h(0)}\hat{\tau}^{-1}$ acts on $\ee_\alpha\in\C_\eps[L]$ by multiplication with $\e^{(2\pi\i)\langle\tau h-h',\alpha\rangle}\eta_\tau(\nu\tau^{-1}\alpha)\eta_\nu(\tau^{-1}\alpha)/\eta_\tau(\tau^{-1}\alpha)/\eta_\nu(\alpha)$ for all $\alpha\in L$. This defines a homomorphism $L\to\{\pm1\}$. Suppose that $\alpha\in L^\nu$. Then, since $\hat\nu$ is a standard lift, the homomorphism becomes $1$. This shows that $\langle f,\alpha\rangle\in\Z$ for all $\alpha\in L^\nu$ or equivalently that $\pi_\nu(f)\in (L^\nu)'=\pi_\nu(L')$.
\end{proof}
The above result shows that all the finite-order automorphisms in $\Aut(V_L)$ can be conjugated into $\Aut(V_L)_{\{\hh\}}=\O(\hat{L})\cdot T$ or more precisely into $\Aut(V)_{\{\Delta\}}=\{g\in\O(\hat{L})\mid \bar{g}\in H_\Delta\}\cdot T$ (cf.\ comment before \autoref{lem:auts}).

Two conjugate automorphisms in $\Aut(V_L)$ must project to the same conjugacy class in $H_\Delta$, but we cannot exclude the possibility that $\hat{\nu}\,\e^{(2\pi\i)h(0)}$ and $\hat{\nu}\,\e^{(2\pi\i)h'(0)}$ for representatives $h$ and $h'$ from different orbits are still conjugate.

The special case of \autoref{thm:latconj} for the Leech lattice \voa{} $V_\Lambda$ was already proved in \cite{MS23}.

\begin{rem}\label{rem:latalgconj}
It is not difficult to formulate \autoref{thm:latconj} for algebraic conjugacy classes, i.e.\ conjugacy classes of cyclic subgroups. In that case, not surprisingly, it suffices to let $\nu$ be from a fixed set of representatives of the algebraic conjugacy classes of $H_\Delta\subseteq\O(L)$. Moreover, we may replace the action of the centraliser $C_{\O(L)}(\nu)$ by the action of the normaliser $N_{\O(L)}(\langle\nu\rangle)$.

More precisely, suppose $\nu$ has order $m$ and let $\tau\in N_{\O(L)}(\langle\nu\rangle)$. Then $\tau\nu\tau^{-1}=\nu^i$ for some $i\in\Z_m$ with $(i,m)=1$ and $\hat{\nu}\,\e^{(2\pi\i)h(0)}$ is algebraically conjugate to $\hat{\nu}\,\e^{(2\pi\i)(i^{-1}\tau h)(0)}$ where $i^{-1}$ denotes the inverse of $i$ modulo $m$.
\end{rem}

\medskip

We conclude this section by describing the orders of the automorphisms in $\Aut(V_L)$. Let $\nu\in\O(L)$ of order $m$. Then the order of an automorphism $g=\hat\nu\,\e^{-(2\pi\i)h(0)}\in\Aut(V_L)$, where $h\in\h^\nu$ so that $\hat\nu$ and $\e^{-(2\pi\i)h(0)}$ commute, must be a multiple of $m$. Of particular relevance in this text will be automorphisms $g$ having the same order as $\nu$. There are two cases to consider, depending on whether $\nu$ exhibits order doubling or not.

First, suppose that $\nu$ does not have order doubling. In this case $\hat\nu$ is of order~$m$ and any element $h$ in $(1/m)L'\cap\h^\nu=(1/m)(L')^\nu=(1/m)(\pi_\nu(L))'$ will yield an inner automorphism $\e^{-(2\pi\i)h(0)}$ of order dividing $m$ so that $g=\hat\nu\,\e^{-(2\pi\i)h(0)}$ has order $m$.

On the other hand, if $\nu$ exhibits order doubling, then $\hat\nu$ has order $2m$ with $\hat\nu^m\,\ee_\alpha=(-1)^{m\langle\pi_\nu(\alpha),\pi_\nu(\alpha)\rangle}\ee_\alpha=(-1)^{\langle\alpha,\nu^{m/2}\alpha\rangle}\ee_\alpha$ for all $\alpha\in L$.
However, there exists a vector $s$ in $(1/(2m))(L')^\nu=(1/(2m))(\pi_\nu(L))'$ defining an inner automorphism $\e^{-(2\pi\i)s(0)}$ of order $2m$ such that $\hat\nu\,\e^{-(2\pi\i)s(0)}$ has order $m$.

Again, we may multiply by $\e^{-(2\pi\i)h(0)}$ for any $h\in(1/m)(L')^\nu=(1/m)(\pi_\nu(L))'$ and obtain an automorphism of order $m$. In total, we have proved:
\begin{prop}\label{prop:sameorder}
Let $L$ be a positive-definite, even lattice and $\hat\nu$ a standard lift of $\nu\in\O(L)$. Then an automorphism $g=\hat\nu\,\e^{-(2\pi\i)h(0)}$ of $V_L$ with $h\in\h^\nu$ has order $m=|\nu|$ if and only if
\begin{equation*}
h\in\begin{cases}(1/m)(L')^\nu&\text{if }|\hat\nu|=m,\\s+(1/m)(L')^\nu&\text{if }|\hat\nu|=2m\end{cases}
\end{equation*}
with $s\in(1/(2m))(L')^\nu$ as defined above.
\end{prop}


\section{Orbifold Construction, Dimension and Rank Formulae}\label{sec:orbdimrk}
In this section we recall the cyclic orbifold construction and present tools to compute the dimension and rank of the corresponding weight-one space, in particular for lattice \voa{}s.


\subsection{Cyclic Orbifold Construction}\label{sec:orb}
We summarise the cyclic orbifold theory for holomorphic \voa{}s developed in \cite{EMS20a,Moe16}. It is a tool used to construct new \voa{}s from known ones. Special cases of this orbifold construction have been studied earlier, for example, the construction of the Moonshine module $V^\natural$ from the Leech lattice \voa{} $V_\Lambda$ as orbifold construction of order $2$ \cite{FLM88}.

Let $V$ be a \strathol{} \voa{} (necessarily of central charge $c\in8\N$) and $G=\langle g\rangle$ a finite, cyclic group of automorphisms of $V$ of order $n\in\Ns$.

By \cite{DLM00} there is an up to isomorphism unique irreducible $g^i$-twisted $V$-module $V(g^i)$ for each $i\in\Z_n$. The uniqueness of $V(g^i)$ implies that there is a representation $\phi_i\colon G\to\Aut_\C(V(g^i))$ of $G$ on the vector space $V(g^i)$ such that
\begin{equation*}
\phi_i(g)Y_{V(g^i)}(v,x)\phi_i(g)^{-1}=Y_{V(g^i)}(g v,x)
\end{equation*}
for all $v\in V$, $i\in\Z_n$. This representation is unique up to an $n$-th root of unity. Denote the eigenspace of $\phi_i(g)$ in $V(g^i)$ corresponding to the eigenvalue $\e^{(2\pi\i)j/n}$ by $W^{(i,j)}$. On $V(g^0)=V$ we choose $\phi_0(g)=g$.

By recent results \cite{Miy15,CM16} (and \cite{DM97}) the \fpvosa{} $V^g=W^{(0,0)}$ (also called \emph{orbifold}) is again \strat{}. It has exactly $n^2$ irreducible modules, namely the $W^{(i,j)}$, $i,j\in\Z_n$ \cite{MT04,DRX17}. One can further show that the conformal weight $\rho(V(g))$ of $V(g)$ is in $(1/n^2)\Z$, and we define the type $t\in\Z_n$ of $g$ by $t=n^2\rho(V(g))\pmod{n}$.

In the following assume that $g$ has type~$0$, i.e.\ that $\rho(V(g))\in(1/n)\Z$. Then it is possible to choose the representations $\phi_i$ such that the conformal weights satisfy
\begin{equation*}
\rho(W^{(i,j)})\in\frac{ij}{n}+\Z
\end{equation*}
and $V^g$ has fusion rules
\begin{equation*}
W^{(i,j)}\boxtimes W^{(l,k)}\cong W^{(i+l,j+k)}
\end{equation*}
for all $i,j,k,l\in\Z_n$ (see \cite{EMS20a}, Section~5), i.e.\ the fusion ring of $V^g$ is the group ring $\C[\Z_n\times\Z_n]$. In particular, all $V^g$-modules are simple currents.

In general, a simple \voa{} $V$ is said to satisfy the \emph{positivity condition} if the conformal weights satisfy $\rho(W)>0$ for any irreducible $V$-module $W\not\cong V$ and $\rho(V)=0$.

Now, if $V^g$ satisfies the positivity condition (it is shown in \cite{Moe18} that this condition is almost automatically satisfied), then the direct sum of $V^g$-modules
\begin{equation*}
V^{\orb(g)}:=\bigoplus_{i\in\Z_n}W^{(i,0)}
\end{equation*}
admits an up to isomorphism unique \strathol{} \voa{} structure extending the given $V^g$-module structure. One calls $V^{\orb(g)}$ the \emph{orbifold construction} associated with $V$ and $g$ \cite{EMS20a}. Two algebraically conjugate automorphisms in $\Aut(V)$ yield isomorphic orbifold constructions. Note that $\bigoplus_{j\in\Z_n}W^{(0,j)}$ is just the old \voa{} $V$.

\medskip

We briefly describe the \emph{inverse} (or \emph{reverse}) \emph{orbifold construction}. Suppose that the \strathol{} \voa{} $V^{\orb(g)}$ is obtained in an orbifold construction as described above. Then via $\amgis v:=\e^{(2\pi\i)i/n}v$ for $v\in W^{(i,0)}$, $i\in\Z_n$, we define an automorphism $\amgis$ of $V^{\orb(g)}$ of order~$n$ and type~$0$, and the unique irreducible $\amgis^j$-twisted $V^{\orb(g)}$-module is $V^{\orb(g)}(\amgis^j)=\bigoplus_{i\in\Z_n}W^{(i,j)}$, $j\in\Z_n$ (see \cite{Moe16}, Theorem~4.9.6). Then
\begin{equation*}
(V^{\orb(g)})^{\orb(\amgis)}=\bigoplus_{j\in\Z_n}W^{(0,j)}\cong V,
\end{equation*}
i.e.\ orbifolding with $\amgis$ is inverse to orbifolding with $g$.


\subsection{Dimension Formula}
Continuing in the setting of the previous subsection, we recall a formula for the dimension of $V^{\orb(g)}_1$, the weight-one Lie algebra of the orbifold construction. In contrast to the other results in this section, this dimension formula is particular to central charge $24$:
\begin{thm}[Dimension Formula, \cite{MS23}]\label{thm:dimform}
Let $V$ be a \strathol{} \voa{} of central charge $24$ and $g$ an automorphism of $V$ of order $n>1$ and type~$0$. Assume that $V^g$ satisfies the positivity condition. Then
\begin{equation*}
\dim(V_1^{\orb(g)})\leq24+\sum_{d\mid n}c_n(d)\dim(V_1^{g^d})
\end{equation*}
where the $c_n(d)$ are defined by
\begin{equation*}
c_n(d)=\frac{n}{d^2}\,\prod_{p\mid d}\,(-p)\prod_{p\mid(d,\frac{n}{d})}(1-p^{-1})\prod_{p\mid \frac{n}{d}}(1+p^{-1})
\end{equation*}
with $p$ prime, or equivalently by the system of equations $\sum_{d\mid n}c_n(d)(t,d)=n/t$ for all $t\mid n$.
\end{thm}
A special case of this formula for orders $n$ such that the genus of the modular curve $X_0(n)=\overline{\Gamma_0(n)\backslash\H}$ is zero was proved in \cite{EMS20b}. In fact, the orders of the automorphisms appearing in this text happen to satisfy the genus-zero condition.

An automorphism $g\in\Aut(V)$ is called \emph{extremal} if the upper bound in the dimension formula is attained.

The obstruction to extremality is a certain linear combination of the dimensions of the subspaces of the twisted modules $V(g^i)$, $i\neq0\pmod{n}$, with weights strictly between $0$ and $1$. Hence, the following corollary is immediate:
\begin{cor}\label{cor:dimform}
Let $V$ be a \strathol{} \voa{} of central charge $24$ and $g$ an automorphism of $V$ of order $n>1$ and type~$0$. Suppose that the conformal weights of the twisted modules obey $\rho(V(g^i))\geq 1$ for all $i\neq0\pmod{n}$. Then $g$ is extremal and
\begin{equation*}
\dim(V_1^{\orb(g)})=24+\sum_{d\mid n}c_n(d)\dim(V_1^{g^d}).
\end{equation*}
\end{cor}
The automorphisms considered in \autoref{sec:cons} will turn out to satisfy $\rho(V(g^i))\geq 1$ for all $i\neq0\pmod{n}$ so that they are all extremal (see \autoref{prop:226gdh}). But note that in general extremal automorphisms need not satisfy this condition (see the examples in \cite{MS23}).


\subsection{Rank Criterion}\label{sec:rankcrit}
\autoref{cor:dimform} allows us to compute the dimension of the weight-one Lie algebra $V^{\orb(g)}_1$ if the central charge is $24$. In this section we shall discuss how to determine the rank of $V^{\orb(g)}_1$ (in any central charge). More precisely, we derive a sufficient criterion for the \emph{orbifold rank condition}, namely that the rank of $V_1^g$ equals the rank of $V_1^{\orb(g)}$.

First, we recall some probably known results about graded Lie algebras. To this end, let $\g=\bigoplus_{i\in\Z_n}\g_i$ be a reductive, $\Z_n$-graded Lie algebra for some $n\in\Ns$ (see \cite{Kac90}, Section~8.1). Equivalently, let $g$ be an automorphism of $\g$ with $g^n=\id$ and the grading of $\g$ obtained as eigenvalue decomposition with respect to $g$, i.e.\ $\g_i=\{x\in\g \mid g(x)=\e^{(2\pi\i)i/n}x\}$ for all $i\in\Z_n$ (and $\g_0=\g^g$).

It is well known that $\g_0$ is also reductive and that $\rk(\g_0)\leq\rk(\g)$ with equality if and only if the automorphism $g$ is inner (see, e.g., Proposition~8.6 in \cite{Kac90}).

We now describe the relationship between Cartan subalgebras of $\g_0$ and $\g$:
\begin{lem}\label{lem:extcart}
Let $\hh_0$ be a Cartan subalgebra of $\g_0$. Then the centraliser $\widetilde{\hh}:=C_{\g}(\hh_0)$ is a Cartan subalgebra of $\g$.
\end{lem}
\begin{proof}
This is stated in Lemma~8.1 in \cite{Kac90} for the case that $\g$ is simple but it is not difficult to see that the result extends to reductive Lie algebras.
\end{proof}
On the other hand:
\begin{lem}\label{lem:fixcart}
Let $\hh$ be a Cartan subalgebra of $\g$. Then $\dim(\hh^g)\leq\rk(\g_0)$ with equality if and only if $\hh^g$ is a Cartan subalgebra of $\g_0=\g^g$.
\end{lem}
\begin{proof}
To prove the inequality we note that for any Cartan subalgebra $\hh$ of $\g$, the fixed points $\hh^g$ are contained in a Cartan subalgebra of $\g_0=\g^g$.
One implication of the equivalence is trivial. For the other assume that $\dim(\hh^g)=\rk(\g_0)$. Since $\hh^g$ can be extended to a Cartan subalgebra of $\g_0$, it must already be a Cartan subalgebra.
\end{proof}
\begin{rem}\label{rem:fixcart}
Note that given a finite-order automorphism $g$ of a reductive Lie algebra $\g$ it is always possible to find a Cartan subalgebra $\hh$ of $\g$ such that $\hh^g$ is a Cartan subalgebra of $\g_0=\g^g$.

Equivalently, if we fix the Cartan subalgebra $\hh$, any finite-order automorphism~$g$ is conjugate (under an inner automorphism) to an automorphism $g'$ such that $\hh^{g'}$ is a Cartan subalgebra of $\g^{g'}$.
\end{rem}
\begin{proof}
A chosen Cartan subalgebra $\hh_0$ of $\g_0$ can be extended to a Cartan subalgebra $\hh$ of $\g$. Then $\hh_0\subseteq\hh^g\subseteq\g_0$ and the result follows.
\end{proof}
Specifically, if $g$ is of the form $\mu\,\e^{(2\pi\i)\ad_v}$ where $\mu\in\Aut(\g)$ is the standard lift of an outer automorphism in $\Out(\g)$ (preserving the choice of Cartan subalgebra and simple roots) and $v$ is in $\hh^\mu$ (cf.\ Proposition~8.1 in \cite{Kac90}),
then equality in the above lemma holds. This is true in particular for the representatives of the conjugacy classes in \autoref{prop:autconj}, restricted to the Lie algebra $V_1$.

\medskip

We return to the orbifold setting, i.e.\ we let $V$ be a \strathol{} \voa{} and $g$ an automorphism of $V$ of order~$n$ and type~$0$ such that $V^g$ satisfies the positivity condition.

The Lie algebras $\g:=V_1$ and $\tilde{\g}:=V^{\orb(g)}_1$ are both reductive and $\Z_n$-graded with the grading on $\g=\bigoplus_{i\in\Z_n}\g_i$ given by the eigenvalue decomposition with respect to $g$ and the one on $\tilde{\g}=\bigoplus_{i\in\Z_n}\tilde{\g}_i$ with respect to the inverse orbifold automorphism of $g$. Then $\g_0=\tilde{\g}_0=V^g_1$, which is also reductive.

We say the \emph{orbifold rank condition} is satisfied if $\rk(V_1^{\orb(g)})=\rk(V_1^g)$.\smallskip

Note that for a subset $S$ of a Lie algebra $\g$ the centraliser is defined as
\begin{align*}
C_{\g}(S)&=\{x\in\g\mid [x,s]=0\text{ for all }s\in S\}\\
&=\{x\in\g\mid \ad_s(x)=0\text{ for all }s\in S\}.
\end{align*}
Now suppose that $M$ is a $\g$-module. Analogously, we define
\begin{equation*}
C_{M}(S)=\{v\in M\mid s\cdot v=0\text{ for all }s\in S\},
\end{equation*}
simply replacing the adjoint action of $\g$ on $\g$ with the module action of $\g$ on $M$.

Let $\hh_0$ be a Cartan subalgebra of $V^g_1$. Then by \autoref{lem:extcart} we know that the centraliser $\widetilde{\hh}:=C_{\tilde{\g}}(\hh_0)$ is a Cartan subalgebra of $\tilde{\g}$.
\begin{lem}\label{lem:cartcent}
The Cartan subalgebra $\widetilde{\hh}$ of $\tilde{\g}$ satisfies
\begin{equation*}
\widetilde{\hh}=\hh_0\oplus\bigoplus_{i\in\Z_n\setminus\{0\}}C_{{\tilde{\g}}_i}(\hh_0).
\end{equation*}
\end{lem}
\begin{proof}
Consider
\begin{equation*}
\widetilde{\hh}=C_{\tilde{\g}}(\hh_0)=C_{\bigoplus_{i\in\Z_n}\tilde{\g}_i}(\hh_0),
\end{equation*}
which, because the Lie bracket respects the $\Z_n$-grading, equals
\begin{equation*}
\bigoplus_{i\in\Z_n}C_{\tilde{\g}_i}(\hh_0)=C_{\g_0}(\hh_0)\oplus\bigoplus_{i\in\Z_n\setminus\{0\}}C_{\tilde{\g}_i}(\hh_0)=\hh_0\oplus\bigoplus_{i\in\Z_n\setminus\{0\}}C_{\tilde{\g}_i}(\hh_0).
\end{equation*}
In the last step we used that as Cartan subalgebra $\hh_0$ is self-centralising.
\end{proof}
The lemma allows us to determine a Cartan subalgebra and hence the rank of $\tilde{\g}$ by only considering the module action of $V^g$ on the orbifold construction $V^{\orb(g)}$ rather than the full Lie algebra structure of $\tilde{\g}$, which would involve intertwining operators between submodules of the twisted $V$-modules $V(g^i)$ for $i\in\Z_n$.

In the same way that $[u,v]:=u_0v$ for $u,v\in V_1$ defines a Lie bracket on $V_1$, does $u\cdot v:=u_0v$ for $u\in V_1$ and $v\in W$ define a Lie algebra module of $V_1$ on any (untwisted) $V$-module $W$ and on any graded component of $W$.

Finally, note that the twisted $V$-modules $V(g^i)$ for $i\in\Z_n$ are untwisted $V^g$-modules, the graded components $V(g^i)_1$ are thus Lie algebra modules of $V^g_1$, and therefore the expression $C_{V(g^i)_1}(\hh_0)$ is well-defined. 

We state a sufficient criterion for the orbifold rank condition:
\begin{prop}[Rank Criterion]\label{prop:rankcrit}
Let $\hh_0$ be a Cartan subalgebra of $V^g_1$. Furthermore, assume that $C_{V(g^i)_1}(\hh_0)=\{0\}$ for all $i\in\Z_n\setminus\{0\}$. Then $\hh_0$ is a Cartan subalgebra of $\tilde{\g}:=V^{\orb(g)}_1$, i.e.\ the orbifold rank condition $\rk(V^{\orb(g)}_1)=\rk(V_1^g)$ is satisfied.
\end{prop}
\begin{proof}
Clearly, $C_{\tilde{\g}_i}(\hh_0)\subseteq C_{V(g^i)_1}(\hh_0)$ since $W^{(i,0)}\subseteq V(g^i)$. Then \autoref{lem:cartcent} implies the assertion.
\end{proof}
\begin{rem}\label{rem:rankcrit}
Note that $C_{V(g^i)_1}(\hh_0)=\{0\}$ for all $i\in\Z_n\setminus\{0\}$ with $(i,n)=1$ is a necessary criterion for the orbifold rank condition to be satisfied.
\end{rem}


\subsection{Lattice \VOA{}s}\label{sec:dimranklat}
We now specialise to the case where $V=V_L$ is the holomorphic \voa{} associated with a positive-definite, even, unimodular lattice $L$ and let $g$ be an automorphism of $V_L$ of order~$n$ and type~$0$. By \autoref{thm:latconj} we may assume that $g$ is of the form $g=\hat{\nu}\,\e^{-(2\pi\i)h(0)}$ for some $\nu\in\O(L)$ and some $h\in\pi_\nu(L\otimes_\Z\Q)$. The minus sign in the exponent is a convention related to the sign convention for twisted modules (see the remark at the beginning of \autoref{sec:lataut}).

\medskip

In order to apply the dimension formula in \autoref{cor:dimform} we have to compute the conformal weights of the twisted $V_L$-modules, which can be described by combining \cite{DL96,BK04} with Section~5 of \cite{Li96}. Since $L$ is unimodular, $V_L$ is holomorphic and there is a unique $g$-twisted $V_L$-module $V_L(g)$. Suppose that $\nu$ has order~$m$ (necessarily dividing $n$) and \emph{Frame shape} (or \emph{cycle shape}) $\prod_{t\mid m}t^{b_t}$ with $b_t\in\Z$, i.e.\ the extension of $\nu$ to $\h$ has characteristic polynomial $\prod_{t\mid m}(x^t-1)^{b_t}$. Note that $\sum_{t\mid m}tb_t=\rk(L)$. Then the conformal weight of $V_L(g)$ is given by
\begin{equation*}
\rho(V_L(g))=\frac{1}{24}\sum_{t\mid m}b_t(t-1/t)+\min_{\alpha\in\pi_\nu(L)+h}\langle\alpha,\alpha\rangle/2\geq0.
\end{equation*}
The number
\begin{equation*}
\rho_\nu:=\frac{1}{24}\sum_{t\mid m}b_t(t-1/t)
\end{equation*}
is called the \emph{vacuum anomaly} of $V_L(g)$. Note that $\rho(V_L(g))>0$ for $g\neq\id$, i.e.\ any \fpvosa{} $V_L^g$ of a lattice \voa{} $V_L$ satisfies the positivity condition.

\medskip

We shall apply the rank criterion from \autoref{sec:rankcrit} to $V_L$. A Cartan subalgebra of $\g=(V_L)_1$ is given by $\hh=\{h(-1)\otimes\ee_0\mid h\in\h\}\cong\h$. Then \autoref{lem:fixcart} becomes:
\begin{lem}\label{lem:fixiscart}
The inequality $\rk((V_L^g)_1)\geq\dim(\h^\nu)=\rk(L^\nu)$ holds; with equality if and only if
\begin{equation*}
\{h(-1)\otimes\ee_0\mid h\in\h^\nu\}\cong\h^\nu
\end{equation*}
is a Cartan subalgebra of $\g_0=(V_L^g)_1$.
\end{lem}
We note that if $\nu\in H_\Delta\subseteq\O(L)$ (see \autoref{sec:autlat}), then equality in the lemma holds. This is in particular true for the representatives of the conjugacy classes in \autoref{thm:latconj}. This follows from the corresponding statement after \autoref{rem:fixcart}. 

To formulate \autoref{prop:rankcrit} in the special case of lattice \voa{}s, recall from \autoref{sec:autlat} that
\begin{equation*}
\hat{\nu}^i=\widehat{(\nu^i)}\e^{-(2\pi\i)s_i(0)}
\end{equation*}
for some vector $s_i\in(1/2)(L^{\nu^i})'\subseteq\h^{\nu^i}$ where $\widehat{(\nu^i)}$ is a standard lift of $\nu^i$. Consequently,
\begin{equation*}
g^i=\widehat{(\nu^i)}\e^{-(2\pi\i)(s_i+ih)(0)}
\end{equation*}
and the unique irreducible $g^i$-twisted $V_L$-module \cite{DL96,BK04,Li96} is of the form
\begin{equation*}
V_{L}(g^i)=M_{\hat\h}(1)[\nu^i]\otimes\ee_{s_i+ih}\C[\pi_{\nu^i}(L)]\otimes\C^{d(\nu^i)}
\end{equation*}
with a grading by the lattice coset $s_i+ih+\pi_{\nu^i}(L)$ and defect $d(\nu^i)\in\Ns$.
\begin{prop}[Rank Criterion]\label{prop:latrankcrit}
Suppose that $\rk(V_L^g)_1=\rk(L^\nu)$ (i.e.\ that $\hh_0:=\{h(-1)\otimes\ee_0\mid h\in\h^\nu\}\cong\h^\nu$ is a Cartan subalgebra of $(V_L^g)_1$). Furthermore assume that $\pi_\nu(s_i)+ih\notin\pi_\nu(L)$ for all $i\in\Z_n\setminus\{0\}$. Then the orbifold rank condition $\rk(V^{\orb(g)}_1)=\rk((V_L^g)_1)$ is satisfied.
\end{prop}
\begin{proof}
Let $v=k(-1)\otimes\ee_0$ with $k\in\h^\nu$ be in the Cartan subalgebra $\hh_0$. Let $0\neq w\in V_{L}(g^i)$ be in the homogeneous space of degree $s_i+ih+\beta$ for the lattice coset grading for some $\beta\in\pi_{\nu^i}(L)$. Then
\begin{equation*}
v\cdot w=v_0w=k(0)w=\langle k,s_i+ih+\beta\rangle w.
\end{equation*}
Now suppose that $w\in C_{V_L(g^i)}(\hh_0)$, i.e.\ that
\begin{equation*}
0=\langle k,s_i+ih+\beta\rangle=\langle k,\pi_\nu(s_i)+ih+\pi_\nu(\beta)\rangle
\end{equation*}
for all $k\in\h^\nu$. Since $\langle\cdot,\cdot\rangle$ is non-degenerate on $\h^\nu$, this is equivalent to $\pi_\nu(s_i)+ih+\pi_\nu(\beta)=0$. If $\pi_\nu(s_i)+ih\notin\pi_\nu(L)$, then this equation cannot be satisfied and hence $C_{V_L(g^i)}(\hh_0)=\{0\}$. In particular, $C_{V_L(g^i)_1}(\hh_0)=\{0\}$ and by \autoref{prop:rankcrit} the assertion follows.
\end{proof}
In the above proposition we demanded that $C_{V_L(g^i)}(\hh_0)=\{0\}$. Of course, we can weaken the assumption by only requiring that $C_{V_L(g^i)_1}(\hh_0)=\{0\}$, resulting in the slightly more technical condition that there is no $\alpha\in s_i+ih+\pi_{\nu^i}(L)$ with $\pi_\nu(\alpha)=0$ and $\rho_{\nu^i}+\langle\alpha,\alpha\rangle/2+k(i,m)/m=1$ for some $k\in\N$.

It will turn out that the automorphisms in \autoref{sec:cons} satisfy the assumptions of \autoref{prop:latrankcrit} and hence the orbifold rank condition so that we can easily determine the rank of the weight-one Lie algebra $V^{\orb(g)}_1$ in these cases.

\begin{rem}\label{rem:latrankcrit}
Reversely, if the orbifold rank condition $\rk((V_L^{\orb(g)})_1)=\rk((V_L^g)_1)$ is satisfied, then, by \autoref{rem:rankcrit},
$1-\rho_\nu\notin(1/m)\N$ or $s_i+ih\notin\pi_\nu(\Lambda)$ for all $i\in\Z_n\setminus\{0\}$ with $(i,n)=1$.
\end{rem}

\begin{rem}\label{rem:orbiextremal}
Finally, note that if $L$ is of rank $24$, $\nu\in H_\Delta\subseteq\O(L)$ (which, as mentioned above, implies that $\hh_0$ is a Cartan subalgebra of $(V_L^g)_1$) and $\nu$ has order $m=n$, then the condition $\pi_\nu(s_i)+ih\notin\pi_\nu(L)$ for all $i\in\Z_n\setminus\{0\}$ in \autoref{prop:latrankcrit} does not only imply the orbifold rank condition but also $\rho(V_L(g^i))\geq1$ for all $i\in\Z_n\setminus\{0\}$ and hence the the extremality of $g$ by \autoref{cor:dimform}. Indeed, since $g$ is of type~$0$, $\rho(V_L(g^i))\in((i,n)/n)\Z$, and $\rho_{\nu^i}\geq 1-(i,m)/m=1-(i,n)/n$ for all $\nu\in H_\Delta$ by an explicit calculation of all vacuum anomalies $\rho_\nu$. Then, $s_i+ih\notin\pi_{\nu^i}(L)$ since $\pi_\nu(s_i)+ih\notin\pi_\nu(L)$ and $\rho(V_L(g^i))$ must be greater than $1-(i,n)/n$ and consequently at least $1$.
\end{rem}


\section{Schellekens' List}\label{sec:schellekens}
In this section we review the classification of \strathol{} \voa{}s. By work of Zhu~\cite{Zhu96} the central charge $c$ of such a \voa{} $V$ is a non-negative multiple of $8$ and its weight-one Lie algebra $V_1$ is reductive \cite{DM04b}.

The most interesting case is that of central charge $24$, in particular because of its connection to bosonic string theory \cite{Bor92}. In 1993 Schellekens proved that the weight-one space $V_1$ must be one of $71$ semisimple or abelian Lie algebras of rank at most $24$, which are listed in Table~1 of \cite{Sch93} and called \emph{Schellekens' list} (see also \cite{DM04,DM06b,Hoe17,EMS20a}).

Examples of such \voa{}s are the \voa{}s $V_N$ associated with the $24$ Niemeier lattices $N$, i.e.\ the positive-definite, even, unimodular lattices of rank $24$ (of which the Leech lattice~$\Lambda$ is the unique one without roots). These are exactly the cases in which the weight-one Lie algebra has rank~$24$ \cite{DM04b}. At the other end of the spectrum is the famous Moonshine module $V^\natural$ \cite{FLM88} with $V_1^\natural=\{0\}$.

If $V_1$ is one of the $69$ semisimple Lie algebras in central charge $24$, then the (simple) affine \voa{} $\langle V_1\rangle\cong L_{\hat\g_1}(k_1,0)\otimes\cdots\otimes L_{\hat\g_s}(k_s,0)$ generated by $V_1$ (with levels $k_i\in\Ns$) is a full \vosa{} of $V$.
Since $ L_{\hat\g_1}(k_1,0)\otimes\cdots\otimes L_{\hat\g_s}(k_s,0)$ is rational,
$V$ decomposes into the direct sum of finitely many irreducible $\langle V_1\rangle$-modules. Schellekens also showed that the isomorphism type of $V_1$ uniquely determines the \voa{} $\langle V_1\rangle$ and the decomposition of $V$ into irreducible $\langle V_1\rangle$-modules up to outer automorphisms of $V_1$.

\medskip

In a combined effort by numerous authors (see \cite{FLM88,DGM90,Don93,DGM96,Lam11,Miy13b,LS12a,LS15a,SS16,LS16a,EMS20a,LS16,Moe16,LL20} for the existence and \cite{DM04,LS15a,LS19,KLL18,LL20,EMS20b,LS20,LS20b} for the uniqueness) the following classification result was proved:
\begin{thm}[Classification]\label{thm:classalt}
Up to isomorphism there are exactly $70$ \strathol{} \voa{}s $V$ of central charge $24$ with $V_1\neq\{0\}$. Such a \voa{} is uniquely determined by the Lie algebra $V_1$.
\end{thm}
While the zero Lie algebra is realised by the Moonshine module $V^\natural$, as of now it is not known whether $V^\natural$ is the unique \strathol{} \voa{} $V$ of central charge $24$ with $V_1=\{0\}$.

Unfortunately, the proof of the above theorem uses a variety of different methods and one might be led to believe that Schellekens' classification is largely sporadic. That this is not the case was shown in two independent papers by the authors of this text \cite{Hoe17,MS23} (see also \cite{CLM22}) who described different uniform constructions of the $70$ \strathol{} \voa{}s $V$ of central charge $24$ with $V_1\neq\{0\}$.

In both works it was observed that the $70$ \voa{}s fall into $11$ families, listed in \autoref{table:11}, associated with certain (algebraic) conjugacy classes of automorphisms in $\Co_0\cong\O(\Lambda)$, the automorphism group of the Leech lattice~$\Lambda$, or equivalently with certain lattice genera (see \autoref{sec:orbit}). We note that the $11$ conjugacy classes are uniquely characterised by their Frame shapes. A complete list of the \voa{}s in each family can be found in the appendix.

\begin{table}[ht]\caption{The $11$ Frame shapes and lattice genera associated with the \strathol{} \voa{}s $V$ of central charge $24$ with $V_1\neq\{0\}$.}
\renewcommand{\arraystretch}{1.15}
\begin{tabular}{c|l|c|l|r|r}
Family & $\O(\Lambda)$ & Ord.\ Doubl.\ & Lattice Genus & Class No.\ & No.\ of VOAs\ \\\hline
A & $\sAA$ & no  & $\gAA$ & 24 & 24\\
B & $\sBB$ & no  & $\gBB$ & 17 & 17\\
C & $\sCC$ & no  & $\gCC$ &  6 &  6\\
D & $\sDD$ & yes & $\gDD$ &  2 &  9\\
E & $\sEE$ & no  & $\gEE$ &  5 &  5\\
F & $\sFF$ & no  & $\gFF$ &  2 &  2\\
G & $\sGG$ & no  & $\gGG$ &  2 &  2\\
H & $\sHH$ & no  & $\gHH$ &  1 &  1\\
I & $\sII$ & no  & $\gII$ &  1 &  1\\
J & $\sJJ$ & yes & $\gJJ$ &  1 &  2\\
K & $\sKK$ & yes & $\gKK$ &  1 &  1
\end{tabular}
\label{table:11}
\end{table}

\medskip

In \autoref{sec:cons}, we add another systematic construction of the $70$ \strathol{} \voa{}s $V$ of central charge $24$ with $V_1\neq\{0\}$, namely as orbifold constructions associated with the $24$ Niemeier lattice \voa{}s $V_N$.

Moreover, we use the corresponding inverse orbifold constructions to systematically show that each such \voa{} $V$ is indeed uniquely determined by the Lie algebra $V_1$ (see \autoref{sec:uniqueness}).


\subsection{Leech-Lattice Approach}\label{sec:orbit}
To provide some context, we explain the constructions in \cite{Hoe17} and \cite{MS23}, beginning with the former. Due to recent advancements we are able to formulate some of the conjectural statements in \cite{Hoe17} as theorems.

\medskip

Let $V$ be a \strat{} \voa{}. Then the Lie algebra $V_1$ is reductive \cite{DM04b}.
Let $\g=\g_1\oplus\cdots\oplus\g_s$ be the semisimple part of $V_1$. Then the (in general non-full) \vosa{} $\langle\g\rangle$ of $V$ generated by $\g$ is isomorphic to the simple affine \voa{} $\langle\g\rangle\cong L_{\hat\g_1}(k_1,0)\otimes\cdots\otimes L_{\hat\g_s}(k_s,0)$ with positive integer levels $k_i\in\Ns$ \cite{DM06b}. This entails that $\langle\cdot,\cdot\rangle=k_i(\cdot,\cdot)$ restricted to $\g_i$ where $\langle\cdot,\cdot\rangle$ is the invariant bilinear form on $V$ normalised such that $\langle\vac,\vac\rangle=-1$ and $(\cdot,\cdot)$ the invariant bilinear form on $\g_i$ normalised such that the long roots have norm $2$. Moreover, it is well known that $\langle\g\rangle$ contains the lattice \voa{} $V_{Q_\g}$
(in general a non-full \vosa{} of~$\langle\g\rangle$) where $Q_\g:=\sqrt{k_1}\,Q^l_1\oplus\cdots\oplus\sqrt{k_s}\,Q^l_s$ and $Q^l_i$ is the lattice spanned by the long roots of $\g_i$ normalised to have squared norm $2$ (see, e.g., Corollary~5.7 in \cite{DM06b}).

Let $\hh\cong Q_\g\otimes_\Z\C$ be the standard Cartan subalgebra of $(V_{Q_\g})_1\subseteq\g$, which is also a Cartan subalgebra of $\g$.
Then the \vosa{} $H:=\langle\hh\rangle$ of $V$ generated by $\hh$ is a Heisenberg \voa{} and $H$ is a full \vosa{} of $V_{Q_\g}$.
We then consider the commutant (or centraliser) $W:=\Com_V(H)=\Com_V(V_{Q_\g})$.
The double commutant $\Com_V(W)=\Com_V(\Com_V(H))$ must be a lattice \voa{}, extending $V_{Q_\g}$, i.e.\ $\Com_V(W)=V_L$ for some lattice extension $L\supseteq Q_\g$, and $W$ is \strat{} by Section~4.3 in \cite{CKLR19}. Note that $V_L$ and $W$ form a dual (or Howe or commuting) pair in $V$, i.e.\ they are their mutual commutants, and $V$ is isomorphic to a (full) \voa{} extension of $W\otimes V_L$.

Now, additionally suppose that $V$ is holomorphic. Then the central charge $c$ of $V$ is a non-negative multiple of $8$, and the rank $r$ of $V_1$ is bounded from above by~$c$. Let us also assume that $V_1=\g$ is semisimple, i.e.\ has no abelian part. Then $\hh$ is a Cartan subalgebra of $V_1$ and $H\subseteq V_{Q_\g}\subseteq V_L$ all have central charge $r$ while $W$ has central charge $c-r$. By construction, $W_1=\{0\}$. The main results in \cite{CKM22,Lin17} and the fact that $V$ is holomorphic imply that there is a braid-reversed equivalence between the module categories of $W$ and~$V_L$. More concretely, since the irreducible $V_L$-modules, the non-integral parts of their conformal weights and the fusion rules are parametrised by the finite quadratic space $L'/L$, the irreducible $W$-modules are also simple currents and along with the non-integral parts of their conformal weights and the fusion rules parametrised by some finite quadratic space $R(W)$, and there is an isometry $\tau\colon L'/L\longrightarrow \overline{R(W)}$ such that
\begin{equation*}
V\cong\bigoplus_{\alpha+L\in L'/L}W_{\tau(\alpha+L)}\otimes V_{\alpha+L},
\end{equation*}
i.e.\ $V$ is a simple-current extension of $W\otimes V_L$. Here, $\overline{R(W)}$ denotes the finite quadratic space obtained from $R(W)$ by multiplying the quadratic form by $-1$.

We specialise to central charge $c=24$. In that case $V_1$ is one of the $71$ Lie algebras in Schellekens' list \cite{Sch93} and hence either zero, abelian of rank $r=24$ or semisimple \cite{DM04}. We assume that $V_1\neq\{0\}$ or equivalently that $r>0$. If $V_1$ is abelian, then $V$ is isomorphic to the Leech lattice \voa{} $V_\Lambda$ \cite{DM04b} and the results of the above paragraph are true with $L=\Lambda$; so we may assume that $V_1$ is semisimple. Then $\langle V_1\rangle$ is actually a full \vosa{} of $V$.
It is observed in~\cite{Hoe17} that Schellekens' classification implies that the lattice $L=L_\g$, called \emph{orbit lattice}, belongs to one of the $11$ genera listed in~\autoref{table:11} and labelled by letters A to K.

Furthermore, is shown in \cite{Hoe17}, Theorem~4.7, that $W$ has to be isomorphic to one of $11$ \voa{}s of the form $V_{\Lambda_\nu}^{\hat\nu}$ for the conjugacy classes $\nu\in\O(\Lambda)$ listed in \autoref{table:11} and corresponding to the genus of the lattice $L_\g$. Here $\Lambda_\nu=(\Lambda^\nu)^\bot$ denotes the coinvariant lattice and $\hat\nu$ a lift of $\nu$ to $\Aut(V_{\Lambda_\nu})$. Note that since $\nu$ acts fixed-point freely on $\Lambda_\nu$, all its lifts are standard and conjugate. The result in \cite{Hoe17} depended on the conjecture that $R(V_{\Lambda_\nu}^{\hat\nu})\cong R(W)\cong \overline{R(V_{L_\g})}$, which was proved in full in \cite{Lam20} by computing $R(V_{\Lambda_\nu}^{\hat\nu})$, with partial results in \cite{Moe16} and by others.

\smallskip In summary:
\begin{thm}[Leech-Lattice Picture, \cite{Hoe17}]\label{thm:hoehn}
Let $V$ be a \strathol{} \voa{} of central charge $24$ with $V_1\neq\{0\}$. Then $V$ is isomorphic to a simple-current extension of
\begin{equation*}
V_{\Lambda_\nu}^{\hat\nu}\otimes V_{L_\g}
\end{equation*}
where $L_\g\supseteq Q_\g$ is a certain lattice determined by $\g$ in one of the $11$ genera listed in \autoref{table:11} and $\nu$ is from the corresponding conjugacy class in $\O(\Lambda)$. Moreover, $V_{\Lambda_\nu}^{\hat\nu}$ and $V_{L_\g}$ form a dual pair in $V$.
\end{thm}
Since every \voa{} of the form $V_{\Lambda_\nu}^{\hat\nu}\otimes V_L$ where $L$ is in the genus associated with $\nu$ may be extended to a \strathol{} \voa{} of central charge $24$, the lattices $L_\g$ where $\g$ runs through the Lie algebras on Schellekens' list exhaust all isomorphism classes of lattices in the $11$ genera.

\medskip

The existence of a \strathol{} \voa{} $V$ of central charge $24$ with $V_1\cong\g$ for each of the $70$ non-zero weight-one Lie algebras~$\g$ on Schellekens' list \cite{Sch93} (see \autoref{thm:classalt}) is then proved by
realising $V$ as holomorphic extension of $V_{\Lambda_\nu}^{\hat\nu}\otimes V_{L_\g}$
(see Theorem~4.4 in \cite{Hoe17}).

In all families except D and J the lattices $L_\g$ are non-isomorphic for different Lie algebras $\g$. For genera D and J the Lie algebra $\g=V_1$ (and thus the \voa{} $V$) additionally depends on the choice of the isometry $\tau\colon L_\g'/L_\g\longrightarrow\overline{R(V_{\Lambda_\nu}^{\hat\nu})}$; see the last two columns in \autoref{table:11}.

\medskip

In principle, the approach in \cite{Hoe17} can also be used to prove the uniqueness of a \strathol{} \voa{} of central charge $24$ with a given Lie algebra $\g=V_1\neq\{0\}$ (see \autoref{thm:classalt}). By \autoref{thm:hoehn} it suffices to classify the holomorphic extensions of $V_{\Lambda_\nu}^{\hat\nu}\otimes V_{L_\g}$ for all orbit lattices in the $11$ genera. The general approach is described in \cite{Hoe17}, Theorem~4.1. Partial results are obtained in Section~4.3 of \cite{Hoe17}. Except for genera D and J one has to show that each orbit lattice $L_\g$ admits only one holomorphic extension of $V_{\Lambda_\nu}^{\hat\nu}\otimes V_{L_\g}$ up to isomorphism, i.e.\ that the choice of the isometry $\tau$ does not matter.

\medskip

Finally, note that if $V_1=\{0\}$ so that $H=V_{L_\g}\cong\C\,\vac$, then the above decomposition degenerates to $V\cong W\otimes\C\,\vac\cong W$ and becomes ineffective.


\subsection{Generalised-Deep-Hole Approach}\label{sec:gdh}
We first recall the notion of \gdh{} from \cite{MS23} and then explain how it was used in \cite{MS23} to systematically realise the Lie algebras on Schellekens' list by cyclic orbifold constructions starting from the Leech lattice \voa{} $V_\Lambda$.

\begin{defi}[\GDH{}, \cite{MS23}]\label{def:gendeephole}
Let $V$ be a \strathol{} \voa{} of central charge $24$ and $g\in\Aut(V)$ of finite order~$n$. We further assume that $V^g$ satisfies the positivity condition. Then $g$ is called a \emph{\gdh{}} of $V$ if $g$
\begin{enumerate}
\item has type~$0$,
\item is extremal (see \autoref{thm:dimform}) and 
\item satisfies the orbifold rank condition (see \autoref{sec:rankcrit}). 
\end{enumerate}
\end{defi}

Schellekens' list can be realised by orbifold constructions from the Leech lattice \voa{} $V_\Lambda$, thus also proving the existence part of \autoref{thm:classalt}:
\begin{thm}[Generalised-Deep-Hole Construction, \cite{MS23}]
Let $\g$ be one of the $71$ Lie algebras on Schellekens' list. Then there exists a \gdh{} $g\in\Aut(V_{\Lambda})$ such that the corresponding orbifold construction has weight-one Lie algebra $(V_\Lambda^{\orb(g)})_1\cong\g$.
\end{thm}
In addition, it is proved:
\begin{thm}[Holy Correspondence, \cite{MS23}]
The cyclic orbifold construction $g\mapsto V_\Lambda^{\orb(g)}$ defines a bijection between the algebraic conjugacy classes of \gdh{}s $g$ in $\Aut(V_{\Lambda})$ with $\rk((V_\Lambda^g)_1)>0$ and the isomorphism classes of \strathol{} \voa{}s $V$ of central charge $24$ with $V_1\neq\{0\}$.
\end{thm}

Moreover, it is shown that the \gdh{}s of $V_\Lambda$ with $\rk((V_\Lambda^g)_1)>0$ project to the $11$ conjugacy classes in $\O(\Lambda)$ listed in \autoref{table:11} under $\Aut(V_\Lambda)\to\Aut(V_\Lambda)/K\cong\O(\Lambda)$ (in the case of the Leech lattice $\Lambda$ the Weyl group $W$ is trivial since $\Lambda$ has no roots), again grouping the $70$ \strathol{} \voa{}s $V$ of central charge $24$ with $V_1\neq\{0\}$
in the same way as in \cite{Hoe17} into the $11$ families A~to~K.

The exactly $38$ algebraic conjugacy classes of \gdh{}s $g$ in $V_\Lambda$ with $\rk((V_\Lambda^g)_1)=0$ all yield the Moonshine module $V^\natural$ with $V^\natural_1=\{0\}$ in the orbifold construction \cite{Car18}.

\medskip

The two pictures from \cite{Hoe17,MS23} are related as follows. First, note that for any automorphism $\nu\in\O(\Lambda)$, the fixed-point (or invariant) sublattice $\Lambda^\nu$ and the coinvariant sublattice $\Lambda_\nu=(\Lambda^\nu)^\bot$ are their mutual orthogonal complements in the Leech lattice $\Lambda$, and $\Lambda$ is a finite-index extension of $\Lambda^\nu\oplus\Lambda_\nu$. Hence, $V_{\Lambda^\nu}\otimes V_{\Lambda_\nu}$ is a dual pair in $V_\Lambda$, i.e.\ $V_{\Lambda^\nu}$ and $V_{\Lambda_\nu}$ are their mutual commutants in $V_\Lambda$.

Let $g\in\Aut(V_\Lambda)$ be of finite order and type~$0$ and let $V\cong V_\Lambda^{\orb(g)}$ be the corresponding orbifold construction. Up to conjugacy, $g$ is of the form $g=\hat\nu\,\e^{-(2\pi\i)h(0)}$ with $\nu\in\O(\Lambda)$ and $h\in\pi_\nu(\Lambda\otimes_\Z\Q)$. As such a $g$ acts on the two tensor factors of $V_{\Lambda^\nu}\otimes V_{\Lambda_\nu}$ separately, $V_{\Lambda^{\nu,h}}\otimes V_{\Lambda_\nu}^{\hat\nu}$ is a full \vosa{} of $V_\Lambda^g$ where $\Lambda^{\nu,h}:=\{\alpha\in\Lambda^\nu\mid\langle\alpha,h\rangle\in\Z\}$ and $\hat\nu$ is a standard lift of the restriction of $\nu$ to $\Lambda_\nu$ (recall that $\nu$ acts fixed-point freely on $\Lambda_\nu$ so that all lifts of $\nu$ are standard and conjugate). In fact, $V_{\Lambda^{\nu,h}}\otimes V_{\Lambda_\nu}^{\hat\nu}$ is a dual pair in the \fpvosa{} $V_\Lambda^g$.

Now, let $V$ be a \strathol{} \voa{} $V$ of central charge $24$ with $\g:=V_1\neq\{0\}$ and $g$ the corresponding \gdh{} of $V_\Lambda$ with $\rk((V_\Lambda^g)_1)=\rk(\Lambda^\nu)>0$. Then $V_\Lambda^{\orb(g)}\cong V$. Moreover, it follows from the classification result in \cite{MS23} that $\nu$ is one of the $11$ automorphisms in \autoref{table:11} and that the weight-one Lie algebra $\g\cong(V_\Lambda^{\orb(g)})_1$ is such that the orbit lattice $L_\g$ is in the genus corresponding to $\nu$. Then, by the results in \cite{Hoe17}, $V\cong V_\Lambda^{\orb(g)}$ is an extension of the dual pair $V_{L_\g}\otimes V_{\Lambda_\nu}^{\hat\nu}$ (with the same $\nu$).

Finally, consider the inverse orbifold construction, which is given by the inner automorphism $\sigma=\e^{-(2\pi\i)u(0)}\in\Aut(V)$ for a certain $u\in L_\g\otimes_\Z\Q$ described in \cite{ELMS21}, i.e.\ $V^{\orb(\sigma_u)}\cong V_\Lambda$. Then also $V_{L_\g^u}\otimes V_{\Lambda_\nu}^{\hat\nu}$ with $L_\g^u:=\{\alpha\in L_\g\mid\langle\alpha,u\rangle\in\Z\}$ forms a dual pair in the \fpvosa{} $V_\Lambda^g\cong V^{\sigma_u}$, which entails that $\Lambda^{\nu,h}\cong L_\g^u$. Therefore, we obtain the following diagram, in which all inclusion arrows denote full \vosa{}s and the vertical arrows are embeddings of dual pairs:
\begin{equation*}
\begin{tikzcd}[sep=large]
\makebox[\widthof{$V$}]{$V_\Lambda$}&&\makebox[\widthof{$V$}]{$V_\Lambda^{\orb(g)}\cong V,$}&\makebox[\widthof{$V$}]{$\quad V_1\cong\g$}\\
\makebox[\widthof{$V$}]{$V_{\Lambda^\nu}\otimes V_{\Lambda_\nu}$}\arrow[hook,"\text{d.p.}"]{u}&\makebox[\widthof{$V$}]{$V_\Lambda^g$}\arrow[hook]{ul}\arrow[hook]{ur}&\makebox[\widthof{$V$}]{$V_{L_\g}\otimes V_{\Lambda_\nu}^{\hat\nu}$}\arrow[hook,"\text{d.p.}"]{u}\\
&\makebox[\widthof{$V$}]{$V_{\Lambda^{\nu,h}}\otimes V_{\Lambda_\nu}^{\hat\nu}$}\arrow[hook]{ul}\arrow[hook,"\text{d.p.}"]{u}\arrow[hook]{ur}&&\makebox[\widthof{$V$}]{$V_{Q_\g}\otimes V_{\Lambda_\nu}^{\hat\nu}$}\arrow[hook]{ul}\\[-11mm]
&\makebox[\widthof{$V$}]{$\cong V_{L_\g^u}\otimes V_{\Lambda_\nu}^{\hat\nu}$}
\end{tikzcd}
\end{equation*}


\section{Systematic Orbifold Construction}\label{sec:cons}
In this section we present a systematic construction of the $70$ \strathol{} \voa{}s $V$ of central charge $24$ with $V_1\neq\{0\}$ (the existence part of \autoref{thm:classalt}), as cyclic orbifold constructions from the $24$ Niemeier lattice \voa{}s $V_N$. We shall also connect this picture to the Leech-lattice approach from~\cite{Hoe17}. Finally, we provide a characterisation in terms of \gdh{}s as defined in~\cite{MS23}.


\subsection{\Good{} Automorphisms}
We shall define certain \emph{\good} automorphisms of the Niemeier lattice \voa{}s. It will turn out that there are exactly $226$ \good{} automorphisms up to algebraic conjugacy and that their corresponding orbifold constructions give all the $70$ \voa{}s in a systematic way. These automorphisms, too, can be divided into $11$ classes corresponding to the $11$ Frame shapes in \autoref{table:11}.

\medskip

Recall that by \autoref{thm:latconj}, given a Niemeier lattice $N$, any finite-order automorphism $g\in\Aut(V_N)$ is conjugate to an automorphism of the form $\hat{\nu}\,\e^{-(2\pi\i)h(0)}$ where $\nu$ is from a list of representatives of the conjugacy classes of $H_\Delta\subseteq\O(N)$ and $h\in\pi_\nu(N\otimes_\Z\Q)$ projects to a list of orbit representatives of the action of $C_{\O(N)}(\nu)$ on $\pi_\nu(N\otimes_\Z\Q)/(N^\nu)'$.

Also recall that $H_\Delta\cong H$ where $H=\O(N)/W$ with Weyl group $W$ and the isometry group of $N$ is a split extension $\O(N)\cong W{:}H$. For ease of presentation, we identify $H$ with $H_\Delta$ in the following.

\begin{defi}[\Good{} Automorphism]\label{defi:good}
Let $N$ be a positive-definite, even, unimodular lattice and $V_N$ the corresponding holomorphic lattice \voa{}. Let $g$ be an automorphism of $V_N$ of finite order $n$ conjugate to $\hat{\nu}\,\e^{-(2\pi\i)h(0)}$ with $\nu\in H$ and $h\in\pi_\nu(N\otimes_\Z\Q)$. Then $g$ is called \emph{\good{}} if
\begin{enumerate}
\item\label{item:good1} $g$ has type~$0$,
\item\label{item:good2} $\nu$ has order $n$ and
\item\label{item:good3} $h\bmod{(N^\nu)'}$ has order $n$.
\end{enumerate}
\end{defi}
We comment on this definition. Item~\eqref{item:good1} guarantees that the orbifold construction exists. Explicitly, it reads
\begin{equation*}
\rho(V_N(g))=\rho_\nu+\min_{\alpha\in (N^\nu)'+h}\frac{\langle\alpha,\alpha\rangle}{2}\in\frac{1}{n}\Z.
\end{equation*}
Item~\eqref{item:good2} can be thought of as a minimality condition on the order of $g$. It means that the order of $g$ equals the order of the projection of $g$ to $\Aut(V_N)/K\cong H$. Explicitly,
\begin{equation*}
h\in\begin{cases}{\frac{1}{n}}N^\nu, &\text{if }|\hat\nu|=n,\\s+\frac{1}{n}N^\nu,&\text{if }|\hat\nu|=2n,\end{cases}
\end{equation*}
with $s\in\frac{1}{2n}N^\nu$ from \autoref{prop:sameorder}. Item~\eqref{item:good3} can be rephrased as the restriction of $g$ to $V_{N^\nu}$ having the same order as $g$ or equivalently $N^{\nu,h}$ being an index-$n$ sublattice of $N^\nu$. The motivation behind this condition will be explained at the end of this section.

We classify all \good{} automorphisms of the Niemeier lattices, including the Leech lattice~$\Lambda$, up to algebraic conjugacy.
\begin{prop}[Classification of \Good{} Automorphisms]\label{prop:226}
There are exactly $226$ algebraic conjugacy classes of \good{} automorphisms of the Niemeier lattice \voa{}s $V_N$, listed in \autoref{table:11x24}.  The Frame shapes of their projections to $\Aut(V_N)/K\cong H$ are given by the $11$ Frame shapes in \autoref{table:11}.
\end{prop}
\begin{proof}
By \autoref{thm:latconj}, for each Niemeier lattice $N$ and for each conjugacy class in $H$ (see \autoref{table:25} in the appendix), represented by $\nu$, we have to determine the orbits under the action of $C_{\O(N)}(\nu)$ on $\pi_\nu(N\otimes_\Z\Q)/(N^\nu)'$. A computer search using \texttt{Magma} \cite{Magma} reveals that there are exactly $230$ orbits (with representatives~$h$), corresponding to at most $230$ non-conjugate automorphisms $g$ in $\Aut(V_N)$. Note that we cannot in principle exclude the possibility that two or more orbit representatives $h$ give conjugate automorphisms $g$ in $\Aut(V_N)$.

By considering the normaliser action (see \autoref{rem:latalgconj}), this reduces to $226$ orbits, corresponding to at most $226$ algebraic conjugacy classes in $\Aut(V_N)$. Finally, by studying some invariants, it is easy to see that these $226$ \good{} automorphisms indeed represent $226$ distinct algebraic conjugacy classes.
\end{proof}

The \good{} automorphisms are listed in \autoref{table:11x24}. The first two columns list the $24$ Niemeier lattices $N$, labelled as in \cite{Hoe17} and by their root lattices. The next $11$ columns are labelled by the Frame shape of the corresponding outer automorphism in $\Aut(V_N)/K\cong H$ considered as element in $\O(N)$ as listed in~\autoref{table:11}. The entry for a Niemeier lattice $N$ and a Frame shape consists of a comma-separated list providing for each corresponding algebraic conjugacy class in $H$ (see \autoref{table:25} in the appendix) the number of \good{} automorphisms in $\Aut(V_N)$. The two last columns and two bottom rows count the total number of algebraic conjugacy classes of \good{} automorphisms and of their projections to $H$.

\begin{table}[ht]\caption{The $226$ \good{} automorphisms of the Niemeier lattice \voa{}s $V_N$.}
\small
\addtolength{\tabcolsep}{-0.1em}
\begin{tabular}{ll|*{11}{r}|rr}
    &                    & \multicolumn{11}{c|}{Outer Automorphism Class}  & \multicolumn{2}{c}{No.\ in}   \\\cline{3-13}
\multicolumn{2}{c|}{$N$} & A & B & C & D & E & F & G & H & I & J & K & $H$ & $\Aut(V_N)$\\\hline
 A1 & $D_{24}$           & 1&.&.&.&.&.&.&.&.&.&.&              1 &  1 \\
 A2 & $D_{16}E_{8}$      & 1&.&.&.&.&.&.&.&.&.&.&              1 &  1 \\
 A3 & $E_{8}^3$          & 1&1&.&.&.&.&.&.&.&.&.&              2 &  2\\
 A4 & $A_{24}$           & 1&.&.&1&.&.&.&.&.&.&.&              2 &  2\\
 A5 & $D_{12}^2$         & 1&.&.&1&.&.&.&.&.&.&.&              2 &  2\\
 A6 & $A_{17}E_{7}$      & 1&2&.&.&.&.&.&.&.&.&.&              2 &  3\\
 A7 & $D_{10}E_{7}^2$    & 1&2&.&.&.&.&.&.&.&.&.&              2 &  3\\
 A8 & $A_{15}D_{9}$      & 1&4&.&.&.&.&.&.&.&.&.&              2 &  5\\
 A9 & $D_{8}^3$          & 1&3&.&.&.&.&.&.&.&.&.&              2 &  4\\
A10 & $A_{12}^2$         & 1&.&.&1&.&.&.&.&.&.&.&              2 &  2\\
A11 & $A_{11}D_{7}E_{6}$ & 1&5&.&.&.&.&.&.&.&.&.&              2 &  6\\
A12 & $E_{6}^4$          & 1&$2,2$&2&.&.&.&1&.&.&.&.&          5 &  8\\
A13 & $A_{9}^2D_{6}$     & 1&4&.&.&2&.&.&.&.&.&.&              3 &  7\\
A14 & $D_{6}^4$          & 1&3&2&1&.&.&.&.&.&.&.&              4 &  7\\
A15 & $A_{8}^3$          & 1&2&.&$1,1$&.&.&.&.&.&.&.&          4 &  5\\
A16 & $A_{7}^2D_{5}^2$   & 1&$6,4,4$&.&.&.&.&.&.&.&.&.&        4 & 15\\
A17 & $A_{6}^4$          & 1&.&3&1&.&.&.&.&.&1&.&              4 &  6\\
A18 & $A_{5}^4D_{4}$     & 1&$4,5$&5&.&4&.&4&.&1&.&.&          7 & 24 \\
A19 & $D_{4}^6$          & 1&$2,3$&$2,3$&1&2&1&$2,3$&.&.&.&.& 10 & 20\\
A20 & $A_{4}^6$          & 1&3&.&$1,1$&5&2&.&.&.&.&1&          7 & 14\\
A21 & $A_{3}^8$          & 1&$3,6$&3&1&7&.&4&1&.&1&.&          9 & 27\\
A22 & $A_{2}^{12}$       & 1&2&4&$1,1$&5&2&4&.&2&$1,2$&1&     12 & 26\\
A23 & $A_{1}^{24}$       & 1&3&2&1&5&2&7&1&2&.&1&             10 & 25\\
A24 & $\Lambda$          & 1&1&1&1&1&1&1&1&1&1&1&             11 & 11\\\hline
\multicolumn{2}{l|}{No.\ in $H$} &24 &24 &10 &15 & 8 & 5 & 8 & 3 & 4 & 5 & 4 & 110 \\
\multicolumn{2}{l|}{No.\ in $\Aut(V_N)$} &24 & 76 & 27 & 15 & 31 & 8 & 26 & 3 & 6 & 6 & 4 &  & 226
\end{tabular}
\label{table:11x24}
\end{table}
Quite remarkably, only outer automorphisms corresponding to the $11$ Frame shapes in \autoref{table:11} admit \good{} automorphisms (cf.\ \autoref{table:25}).

\smallskip We also study the powers of the \good{} automorphisms:
\begin{prop}[Powers of \Good{} Automorphisms]\label{prop:power}
A power of a \good{} automorphism is \good{}. The non-trivial powers of the $226$ \good{} automorphisms are tabulated in \autoref{table:Epower} to \autoref{table:Kpower} in the appendix.
\end{prop}
With the classification in \autoref{prop:226} one may prove this result in a case-by-case computation. However, a more conceptual proof will be given in \autoref{sec:char}.

\smallskip We now formulate our main result:
\begin{thm}[Main Result]\label{thm:main}
Let $g=\hat{\nu}\,\e^{-(2\pi\i)h(0)}$ be one of the $226$ \good{} automorphisms of the Niemeier lattice \voa{}s $V_N$. Then $(V_N^{\orb(g)})_1$ is isomorphic to the Lie algebra tabulated in \autoref{table:A} to \autoref{table:K} in the appendix. 
\end{thm}
We shall prove \autoref{thm:main} in \autoref{sec:char}. We thus obtain a systematic construction of \strathol{} \voa{}s $V$ of central charge $24$ realising all non-zero Lie algebras $\g$ on Schellekens' list as weight-one spaces $V_1\cong\g$ (the existence statement in \autoref{thm:classalt}):
\begin{cor}[Systematic Orbifold Construction]\label{cor:main}
The cyclic orbifold constructions $V_N^{\orb(g)}$, where $N$ runs through the $24$ Niemeier lattices and $g$ through the \good{} automorphisms of the corresponding lattice \voa{}s $V_N$, realise all $70$ non-zero Lie algebras $\g$ on Schellekens' list as weight-one spaces $(V_N^{\orb(g)})_1\cong\g$.
\end{cor}

We note by inspecting \autoref{table:Epower} to \autoref{table:Kpower} in the appendix:
\begin{rem}
Each of the $226$ algebraic conjugacy classes of \good{} automorphisms $g$ is uniquely determined by the algebraic conjugacy class of its projection to $\Aut(V_N)/K\cong H$ and the orbifold constructions $V_N^{\orb(g^d)}$ for all $d$ dividing the order~$n$ of $g$.
\end{rem}


\subsection{Relation with the Leech-Lattice Picture}\label{sec:leechrel}
We describe the connection of our orbifold construction with the Leech lattice description of the \strathol{} \voa{}s of central charge $24$ in \cite{Hoe17} reviewed in \autoref{sec:orbit}.

\medskip

Given a positive-definite, even lattice $L$, there is an induced action of $\O(L)$ on the discriminant form $L'/L$, leading to a short exact sequence
\begin{equation*}
1\longrightarrow\O_0(L)\longrightarrow\O(L)\longrightarrow\overline{\O}(L)\longrightarrow 1
\end{equation*}
where $\overline{\O}(L)$ is the subgroup of $\O(L'/L)$ induced by $\O(L)$ and $\O_0(L)$ are the automorphisms of $L$ that act trivially on $L'/L$.

Now, let $N$ be a Niemeier lattice other than the Leech lattice~$\Lambda$. As mentioned before, $\O(N)\cong W{:}H$. The coinvariant lattice $N_H=(N^H)^\bot$, which is a lattice without norm-two vectors, has an up to conjugation unique $H$-equivariant embedding into $\Lambda$ in such a way that $H\cong\O_0(N_H)\cong\O_0(\Lambda_H)$ (see \cite{Nik80}, Remark~1.14.7 and Proposition~1.14.8, and \cite{HM16}). This embedding allows us to consider an element $\nu\in H$ as an automorphism in $\O(\Lambda)$, which we also denote by~$\nu$.

\medskip

Note that, as in the case of the Leech lattice, $V_{N^\nu}\otimes V_{N_\nu}$ is a dual pair in $V_N$ for any automorphism $\nu\in\O(N)$ where we recall the coinvariant lattice $N_\nu=(N^\nu)^\bot$.

Let $g\in\Aut(V_N)$ be of finite order and type~$0$ and let $V\cong V_N^{\orb(g)}$ be the corresponding orbifold construction. Up to conjugacy, $g$ is of the form $g=\hat\nu\,\e^{-(2\pi\i)h(0)}$ with $\nu\in H$ and $h\in\pi_\nu(N\otimes_\Z\Q)$. Then the \voa{} $V_{N^{\nu,h}}\otimes V_{N_\nu}^{\hat\nu}$ with $N^{\nu,h}:=\{\alpha\in N^\nu\mid \langle\alpha,h\rangle\in\Z\}$ is a full \vosa{} of the \fpvosa{} $V_N^g$ and of $V_N^{\orb(g)}$.
In fact, $V_{N^{\nu,h}}\otimes V_{N_\nu}^{\hat\nu}$ is a dual pair in $V_\Lambda^g$. Note that $N_\nu\cong\Lambda_\nu$. Since $\nu$ acts fixed-point freely on $N_\nu$ and $\Lambda_\nu$, all its lifts are standard and conjugate so that $V_{N_\nu}^{\hat\nu}\cong V_{\Lambda_\nu}^{\hat\nu}$.

For the \good{} automorphisms, by \autoref{prop:226}, $\nu$ is from one of the $11$ algebraic conjugacy classes of $\O(\Lambda)$ in \autoref{table:11}. On the other hand, by \cite{Hoe17}, given a \strathol{} \voa{} $V$ of central charge $24$, \hbox{$V_{L_\g}\otimes V_{\Lambda_\mu}^{\hat\mu}$} forms a dual pair in $V$ where $L_\g$ is the orbit lattice associated with $\g=V_1$ and $\mu$ is the one of the $11$ algebraic conjugacy classes in \autoref{table:11} corresponding to the genus of $L_\g$.

Not surprisingly, all \good{} automorphisms have the property that $\mu=\nu$, i.e.\ the orbifold Lie algebra $\g$ we determine in \autoref{thm:main} is such that $V_{L_\g}\otimes V_{\Lambda_\nu}^{\hat\nu}$ forms a dual pair in $V_N^{\orb(g)}$. As in \autoref{sec:gdh}, this allows us to combine the orbifold picture with the Leech-lattice picture from \cite{Hoe17}.
$V_{L_\g}$ must be an extension of $V_{N^{\nu,h}}$ with the same Virasoro element and we get the inclusion $V_{N^{\nu,h}}\otimes V_{\Lambda_\nu}^{\hat\nu}\hookrightarrow V_{L_\g}\otimes V_{\Lambda_\nu}^{\hat\nu}$. In particular, the orbit lattice $L_\g$ is an extension of $N^{\nu,h}$.

In fact, by item~\eqref{item:good3} of their definition, the \good{} automorphisms satisfy that $L_\g\cong N^{\nu,h}$, implying that $V_{N^{\nu,h}}\otimes V_{\Lambda_\nu}^{\hat\nu}$ itself is already a dual pair in $V\cong V_N^{\orb{g}}$, in addition to being a dual pair in $V_N^g$. Indeed, that $N^{\nu,h}$ is an index-$n$ sublattice of $N^\nu$ is equivalent to $L_\g\cong N^{\nu,h}$ since
\begin{equation*}
|L_\g'/L_\g|=|R(V_{\Lambda_\nu}^{\hat\nu})|=n^2|(\Lambda_\nu)'/\Lambda_\nu|=n^2|(N_\nu)'/N_\nu|=n^2|(N^\nu)'/N^\nu|
\end{equation*}
by the results in \cite{Lam20} and because $N^{\nu,h}$ is a sublattice of $L_\g$. This finally motivates item~\eqref{item:good3} in the definition of \good{} automorphisms (\autoref{defi:good}).

We proved:
\begin{prop}\label{prop:leechkurz}
The \good{} automorphisms $g=\hat{\nu}\,\e^{-(2\pi\i)h(0)}$ of the Niemeier lattice \voa{}s $V_N$ and the corresponding orbifold constructions $V_N^{\orb(g)}\cong V$ satisfy the following diagram, in which all inclusion arrows represent full \vosa{}s and the vertical arrows are embeddings of dual pairs. In particular, $N^{\nu,h}\cong L_\g$.
\begin{equation*}
\begin{tikzcd}[sep=large]
\makebox[\widthof{$V$}]{$V_N$}&&\makebox[\widthof{$V$}]{$V_N^{\orb(g)}\cong V,$}&\makebox[\widthof{$V$}]{$\quad V_1\cong\g$}\\
\makebox[\widthof{$V$}]{$V_{N^\nu}\otimes V_{N_\nu}$}\arrow[hook,"\text{d.p.}"]{u}&\makebox[\widthof{$V$}]{$V_N^g$}\arrow[hook]{ul}\arrow[hook]{ur}&\makebox[\widthof{$V$}]{$V_{L_\g}\otimes V_{\Lambda_\nu}^{\hat\nu}$}\arrow[hook,"\text{d.p.}"]{u}\\
&\makebox[\widthof{$V$}]{$V_{N^{\nu,h}}\otimes V_{N_\nu}^{\hat\nu}$}\arrow[hook]{ul}\arrow[hook,"\text{d.p.}"]{u}\arrow[hook]{ur}{\cong}&&\makebox[\widthof{$V$}]{$V_{Q_\g}\otimes V_{\Lambda_\nu}^{\hat\nu}$}\arrow[hook]{ul}
\end{tikzcd}
\end{equation*}
\end{prop}

\begin{rem}
Using \autoref{prop:leechkurz} it is clear why all \strathol{} \voa{}s of central charge $24$ belonging to one of the $11$ genera in \autoref{table:11} can be realised as an orbifold construction of a Niemeier lattice \voa{} by a \good{} automorphism if at least one of the \voa{}s for this genus can be obtained in this way: the \vosa{} $V_{N^{\nu}}\otimes V_{N_\nu}$ of $V_N$ is obtained from the \vosa{} \hbox{$V_{N^{\nu,h}}\otimes V_{N_\nu}^{\hat\nu}$} of $V$ by extending the lattice $N^{\nu,h}$ by a certain isotropic vector $h^*$ in its discriminant form and by extending $V_{N_\nu}^{\hat\nu}\cong V_{\Lambda_\nu}^{\hat\nu}$ to $V_{\Lambda_\nu}$. When $L$ runs through the different lattices in the genus of $N^{\nu,h}$ and is extended by an isotropic vector in the same $\O((N^{\nu,h})'/N^{\nu,h})$-orbit as $h^*$, then the resulting lattices $K$ are all contained in the genus of $N^{\nu}$ and the resulting extensions of $V_K\otimes V_{\Lambda_\nu}$ are lattice \voa{}s associated with a Niemeier lattice.
\end{rem}


\subsection{Generalised-Deep-Hole Characterisation}\label{sec:char}
First, we show that the \good{} automorphisms of the Niemeier lattice \voa{}s are \gdh{}s in the sense of \cite{MS23} (see~\autoref{def:gendeephole}). This will then help us in determining the orbifold construction $V_N^{\orb(g)}$ and proving~\autoref{thm:main}.

\begin{prop}\label{prop:226gdh}
Let $g$ be a \good{} automorphism of a Niemeier lattice \voa{} $V_N$. Then $g$ is a \gdh{}.
\end{prop}
\begin{proof}
Let $n$ denote the order of $g$. We check that all \good{} automorphisms $g$ satisfy $\pi_\nu(s_i)+ih\notin\pi_\nu(N)$ for all $i\in\Z_n\setminus\{0\}$. (For $(i,n)=1$ this is immediate from item~\eqref{item:good3} in \autoref{defi:good} but for the other powers of $g$ we need to check this case by case using the classification result in \autoref{prop:226}.)
As explained in \autoref{rem:orbiextremal}, this not only implies the orbifold rank condition by \autoref{prop:latrankcrit} but also the extremality of $g$ by \autoref{cor:dimform}. Since by definition a \good{} automorphism is of type~$0$, it follows that $g$ is a \gdh{}.
\end{proof}

By the previous result, for each of the $226$ \good{} automorphisms we know the rank and the dimension of $(V_N^{\orb(g)})_1$ by \autoref{cor:dimform} and \autoref{prop:latrankcrit}. Using Schellekens' list of possible weight-one Lie algebras $V_1$ of \strathol{} \voa{}s of central charge $24$, together with a few additional data, we are then able to determine the Lie algebra structure of $(V_N^{\orb(g)})_1$ in each case by exclusion.

\begin{proof}[Proof of \autoref{thm:main}]
In addition to the rank and the dimension of $(V_N^{\orb(g)})_1$ (see \autoref{prop:226gdh}) we know:
\begin{enumerate}[leftmargin=*]
\item\label{item:proof1} By the inverse orbifold construction the Lie algebra $(V_N^g)_1$ is a fixed-point Lie subalgebra of $(V_N^{\orb(g)})_1$ under an automorphism of order dividing the order~$n$ of $g$. On the other hand, the possible fixed-point Lie subalgebras of a Lie algebra on Schellekens' list under finite-order automorphisms can be classified using \cite{Kac90}, Chapter~8.
\item\label{item:proof2} By \autoref{prop:226gdh} we know that $\{h(-1)\otimes\ee_0\mid h\in\h^\nu\}\cong\h^\nu$ is a Cartan subalgebra of $(V_N^{\orb(g)})_1$. It is straightforward to compute the action of the Cartan subalgebra on the left- and right-hand side of the inclusion
\begin{equation*}
(V_N^g)_1\oplus\!\!\bigoplus_{\substack{i\in\Z_n\\(i,n)=1}}\!\!(V_N(g^i))_1\subseteq(V_N^{\orb(g)})_1\subseteq(V_N^g)_1\oplus\bigoplus_{\substack{i\in\Z_n\\i\neq0}}(V_N(g^i))_1.
\end{equation*}
The corresponding (potential) root vectors must be compatible with the root system of $(V_N^{\orb(g)})_1$, including their lengths with respect to the unique non-degenerate, invariant bilinear form on $V_N^{\orb(g)}$ normalised such that $\langle\vac,\vac\rangle=-1$.

Note that the long roots of the root systems of the simple components $\g_i$ of the Lie algebras on Schellekens' list have lengths $2/k_i$ with respect to this bilinear form where the $k_i\in\Ns$ are the levels (see \autoref{sec:orbit}).
\end{enumerate}
Both items together with Schellekens' list are sufficient to reduce the possible Lie algebra structures of $(V_N^{\orb(g)})_1$ to just one case for all $226$ automorphisms $g$.
\end{proof}

We showed in \autoref{prop:226gdh} that the \good{} automorphisms $g$ of the Niemeier lattice \voa{}s are \gdh{}s. By construction, they satisfy that the order of $g$ equals the order of the projection of $g$ to $\Aut(V_N)/K\cong H$. The \good{} automorphisms are in fact characterised by these properties:
\begin{thm}\label{thm:class}
The \good{} automorphisms of the \voa{}s $V_N$ associated with the $24$ Niemeier lattices $N$ are exactly the \gdh{}s $g$ of $V_N$ whose orders equal the orders of their respective projections to $\Aut(V_N)/K\cong H$ (and with $\rk((V_\Lambda^g)_1)>0$ in the case of the Leech lattice $\Lambda$).
\end{thm}

\begin{proof}
In the following we classify the algebraic conjugacy classes of \gdh{}s of the Niemeier lattice \voa{}s whose orders are equal to the orders of their respective projections to $\Aut(V_N)/K\cong H$ (with $\rk((V_\Lambda^g)_1)>0$ in the case of the Leech lattice $\Lambda$). We obtain an upper bound of $226$ classes, which then must be \good{} because there are $226$ classes of \good{} automorphisms by \autoref{prop:226} and all of these are \gdh{}s (with the additional properties) by \autoref{prop:226gdh}.

For each Niemeier lattice $N$ we let $\nu$ vary through the algebraic conjugacy classes in $\Aut(V_\Lambda)/K\cong\O(N)/W=H$ (see \autoref{table:25} in the appendix). This includes but is not limited to those in \autoref{table:11x24}. Then we let $h$ vary through the orbits (computed using \texttt{Magma} \cite{Magma}) of the action of $C_{\O(N)}(\nu)$ on $\pi_\nu(N\otimes_\Z\Q)/(N^\nu)'$ and consider the automorphisms of the form $g=\hat{\nu}\,\e^{-(2\pi\i)h(0)}$ of order~$n$ and type~$0$ (see \autoref{thm:latconj} and \autoref{rem:latalgconj}).

We then eliminate those entries from this finite list that cannot be extremal since $\rho(V_N(g^i))<1$ for some $i\in\Z_n$ with $(i,n)=1$ (see \cite{MS23}) or cannot satisfy the orbifold rank condition $\rk((V_N^{\orb(g)})_1)=\rk((V_N^g)_1)$ by \autoref{rem:latrankcrit} and items \eqref{item:proof1} and \eqref{item:proof2} in the proof of \autoref{thm:main} in combination with Schellekens' list of possible weight-one Lie algebras of \strathol{} \voa{}s of central charge $24$.

This leaves us with $226$ automorphisms of the Niemeier lattice \voa{}s (all corresponding to the $11$ algebraic conjugacy classes in $\O(\Lambda)$ listed in \autoref{table:11}), representing the $226$ distinct algebraic conjugacy classes of \good{} automorphisms.
\end{proof}
Note that the $38$ algebraic conjugacy classes of \gdh{}s $g$ in $V_\Lambda$ with $\rk((V_\Lambda^g)_1)=0$ trivially satisfy the property that the order of the projection of $g$ to $\Aut(V_\Lambda)/K\cong\O(\Lambda)$ is equal to the order of $g$ but they are not \good{} as item~\eqref{item:good3} in \autoref{defi:good} is not satisfied.

\medskip

We conclude this section with the proof of \autoref{prop:power}:
\begin{proof}[Proof of \autoref{prop:power}]
Let $g$ be a \good{} automorphism of a Niemeier lattice \voa{} $V_N$. Then, as stated in the proof of \autoref{prop:226gdh}, $g$ satisfies the assumptions of \autoref{rem:orbiextremal}. It is easy to see that any power $g^i$ of $g$ still satisfies these assumptions so that $g^i$ is a \gdh{} whose projection to $\Aut(V_N)/K$ has the same order as $g^i$. Hence, by \autoref{thm:class}, $g^i$ is \good{}.
\end{proof}


\section{Systematic Uniqueness Proof}\label{sec:uniqueness}
In this section we show that, given a \strathol{} \voa{} $V$ of central charge $24$ with $V_1\neq\{0\}$, the Lie algebra structure of $V_1$ uniquely determines the \voa{} $V$ (uniqueness statement in \autoref{thm:classalt}). Together with the existence result in \autoref{cor:main} this implies that there are exactly $70$ such \voa{}s up to isomorphism. We give the first systematic proof of this statement. We note that we also use the uniqueness
of the decomposition of $V$ into $\langle V_1\rangle$-modules proved in \cite{Sch93}.

The proof follows the strategy proposed in \cite{LS19} to find for all \voa{}s $V$ with a certain Lie algebra $V_1$ an inner automorphism $\amgis\in\Aut(V)$ of type~$0$ such that $V^{\orb(\amgis)}$ is isomorphic to some Niemeier lattice \voa{} $V_N$ and such that the inverse orbifold automorphism $g\in\Aut(V_N)$ is the unique one in $\Aut(V_N)$ up to algebraic conjugacy with certain properties (depending only on $V_1$ and $\amgis$) that allow us to conclude that $V\cong V_N^{\orb(g)}$.

\begin{prop}\label{prop:unique1}
Let $\g$ be one of the $70$ non-zero Lie algebras on Schellekens' list and let $V$ be a \strathol{} \voa{} $V$ of central charge $24$ with $V_1\cong\g$. Then there exists an inner automorphism $\amgis\in\Aut(V)$ such that $V^\amgis$ satisfies the positivity condition and the following holds:
\begin{enumerate}
\item the order $n$ of $\amgis$ equals the order of the element in $\O(\Lambda)$ associated with the family of $\g$ (see \autoref{table:11}),
\item $\amgis$ is of type~$0$,
\item the conformal weights satisfy $\rho(V(\amgis^i))\geq 1$ for all $i\neq0\pmod{n}$, and 
\item $\dim(V^{\orb(\amgis)}_1)=\dim((V_N)_1)$ for some Niemeier lattice $N$.
\end{enumerate}
\end{prop}
We find a total of $226$ such inner automorphisms.
\begin{proof}
Since the case of abelian $\g$ is trivial ($V$ must be isomorphic to the Leech lattice \voa{} $V_\Lambda$ by \cite{DM04b}), we may assume that $\g$ is semisimple. In \autoref{thm:main} we realised each of the $70$ non-zero Lie algebras $\g$ on Schellekens' list as weight-one Lie algebra $\bar{V}_1\cong\g$ of a concrete orbifold realisation $\bar{V}:=V_N^{\orb(\bar{g})}$ associated with a Niemeier lattice $N$ and a \good{} automorphism $\bar{g}$ of order~$n$.

For each of these orbifold constructions the inverse orbifold automorphism $\bar{\amgis}\in\Aut(\bar{V})$ has order~$n$, type~$0$ and satisfies $\bar{V}^{\orb(\bar{\amgis})}=V_N$. Moreover, $\bar{V}^{\bar{\amgis}}$ satisfies the positivity condition. Since $\bar{g}$ has the property that $\rho(V_N(\bar{g}^i))\geq 1$ for all $i\neq0\pmod{n}$ (see proof of \autoref{prop:226gdh}),
it even holds that $\rho(\bar{V}(\bar{\amgis}^i))\geq 1$ for all $i\neq0\pmod{n}$.
As a \gdh{} (see \autoref{prop:226gdh}) $\bar{g}$ satisfies the orbifold rank condition, which is equivalent to $\bar{\amgis}$ being inner (see \autoref{prop:equiv3}).

Now, let $V$ be any \strathol{} \voa{} $V$ of central charge $24$ with $V_1\cong\g\cong\bar{V}_1$. Then the lowest-order trace identity in \cite{Sch93} and the results of \cite{DM04,DM06b} imply that the full \vosa{} $\langle V_1\rangle$ of $V$ is isomorphic to $\langle\bar{V}_1\rangle$. Indeed, they are both isomorphic to the simple affine \voa{} $L_{\hat\g_1}(k_1,0)\otimes\cdots\otimes L_{\hat\g_s}(k_s,0)$ (see \autoref{sec:schellekens}), with the levels $k_i\in\Ns$ determined by said trace identity. Moreover, it is shown in \cite{Sch93} that also the decomposition of $V$ into irreducible modules of $L_{\hat\g_1}(k_1,0)\otimes\cdots\otimes L_{\hat\g_s}(k_s,0)$ is uniquely determined by the Lie algebra structure of $V_1$, i.e.\ $V\cong\bar{V}$ as modules of $\langle V_1\rangle\cong\langle\bar{V}_1\rangle$. In other words, there is a linear isomorphism $\psi\colon\bar{V}\to V$ satisfying $\psi Y(a,z)b=Y(\psi a,z)\psi b$ for all $a\in\langle\bar{V}_1\rangle$ and $b\in\bar{V}$, restricting to a \voa{} isomorphism $\langle\bar{V}_1\rangle\to\langle V_1\rangle$
and a Lie algebra isomorphism $\bar{V}_1\to V_1$.

Crucially, as an inner automorphism, $\bar\amgis$ is specified in terms of the action of $\bar{V}_1$ on $\bar{V}$ via the zero-mode. This allows us to transport the automorphism $\bar\amgis$ to an automorphism of $V$ via $\psi$. By \autoref{prop:equiv2} we may assume that $\bar\amgis=\e^{(2\pi\i)v_0}$ for some $v$ in a choice $\bar\hh$ of Cartan subalgebra of $\bar{V}_1$, though this is not essential. Then $\hh:=\psi(\bar\hh)$ is a Cartan subalgebra of $V_1$. We define $\amgis:=\psi\bar\amgis\psi^{-1}$, which is an (a priori linear) isomorphism of $V$. The identity $\amgis=\psi\e^{(2\pi\i)v_0}\psi^{-1}=\e^{(2\pi\i)(\psi v)_0}$ (where $\psi v\in\hh\subseteq V_1$) shows that $\amgis$ is an inner automorphism and hence a \voa{} automorphism of $V$.

We stated above that $\psi Y(a,z)b=Y(\psi a,z)\psi b$ for all $a\in\langle\bar{V}_1\rangle$ and $b\in\bar{V}$. Here, the left $Y$ denotes the module vertex operator of $\langle\bar{V}_1\rangle$ on $\bar{V}$ and the right one that of $\langle V_1\rangle$ on $V$. For the orbifold construction we consider the irreducible $\bar\amgis^i$-twisted $\bar{V}$-modules $\bar{V}(\bar{\amgis}^i)$ for $i\in\Z_n$, which are also $\bar\amgis^i$-twisted $\langle\bar{V}_1\rangle$-modules (and analogously without the bar). As $\langle\bar{V}_1\rangle$ is a full \vosa{} of $\bar{V}$, the $L_0$-grading of $\bar{V}(\bar{\amgis}^i)$ is already determined by the structure as a $\bar\amgis^i$-twisted $\langle\bar{V}_1\rangle$-module (and again without the bar). On the other hand, the definition of the twisted module vertex operator $Y^{(v)}$ for inner automorphisms $\e^{(2\pi\i)v_0}$ in \cite{Li96} shows that $\psi Y^{(iv)}(a,z)b=Y^{(i\psi v)}(\psi a,z)\psi b$ for all $a\in\langle\bar{V}_1\rangle$ and $b\in\bar{V}(\bar{\amgis}^i)$. This implies that $\bar{V}(\bar{\amgis}^i)$ is isomorphic to $V(\amgis^i)$ as a $\bar\amgis^i$-twisted/$\amgis^i$-twisted module of $\langle V_1\rangle\cong\langle\bar{V}_1\rangle$. In particular, they have the same $L_0$-grading.

We conclude that the automorphisms $\amgis\in\Aut(V)$ and $\bar{\amgis}\in\Aut(\bar{V})$ share the following properties:
\begin{itemize}
\item By definition, $\amgis$ and $\bar\amgis$ have the same order $n$.
\item Since $\psi(\bar{V}_1^{\bar{\amgis}})=V_1^\amgis$, $\amgis$ and $\bar{\amgis}$ have isomorphic fixed-point Lie subalgebras $V_1^\amgis\cong\bar{V}_1^{\bar{\amgis}}\cong (V_N^{\bar{g}})_1$.
\item For $i\in\Z_n$, the conformal weights of the twisted modules $V(\amgis^i)$ and $\bar{V}(\bar{\amgis}^i)$ are the same. Hence, $\amgis$ is also of type~$0$ and $V^\amgis$ satisfies the positivity condition so that the orbifold construction $V^{\orb(\amgis)}$ exists.
\item Moreover, $\rho(V(\amgis^i))=\rho(\bar{V}(\bar{\amgis}^i))\geq 1$ for all $i\neq0\pmod{n}$. 
\item Finally, $\dim(V^{\orb(\amgis)}_1)=\dim(\bar{V}^{\orb(\bar{\amgis})}_1)=\dim((V_N)_1)$ by the dimension formula (see \autoref{cor:dimform}).\qedhere
\end{itemize}
\end{proof}
\begin{rem}
Instead of proving the existence of the inner automorphisms $\amgis$ in $\Aut(V)$ in the above lemma indirectly by mimicking the inverse orbifold automorphisms of the \good{} automorphisms $\bar{g}\in\Aut(V_N)$, one could also explicitly write down an appropriate element $v$ in $\hh$, the Cartan subalgebra of $V_1$, and then show that $\amgis=\e^{(2\pi\i)v_0}\in\Aut(V)$ has the desired properties (using the methods discussed in, e.g., \cite{EMS20b}, Section~7), and we have done so for many of the possible Lie algebras $V_1$. However, we find the above approach to be more conceptual.
\end{rem}

For most of the automorphisms $\amgis$ from \autoref{prop:unique1} it is possible to determine the \voa{} $V^{\orb(\amgis)}$ obtained in the orbifold construction:
\begin{lem}\label{lem:unique2}
For $157$ (at least one for each weight-one Lie algebra $\g$ on Schellekens' list) of the $226$ inner automorphisms $\amgis$ defined in \autoref{prop:unique1} one can show that
\begin{enumerate}
\item[(5)] $V^{\orb(\amgis)}\cong V_N$ for some Niemeier lattice $N$.
\end{enumerate}
\end{lem}
\begin{proof}
Continuing where we left off in the proof of \autoref{prop:unique1}, we attempt to identify the orbifold Lie algebra $V^{\orb(\amgis)}_1$. We know that $\dim(V^{\orb(\amgis)}_1)=\dim((V_N)_1)$ and that $V_1^{\amgis}\cong (V_N^{\bar{g}})_1$ is the fixed-point Lie subalgebra of $V^{\orb(\amgis)}_1$ under an automorphism of order dividing $n$. The possible fixed-point Lie subalgebras of a semisimple (or reductive) Lie algebra under finite-order automorphisms can be classified using \cite{Kac90}, Chapter~8. In $157$ of the $226$ cases we consider, there is only one Lie algebra on Schellekens' list, namely $(V_N)_1$, that has dimension $\dim((V_N)_1)$ and $(V_N^{\bar{g}})_1$ as possible fixed-point Lie subalgebra under an automorphism of order dividing $n$. Hence $V^{\orb(\amgis)}_1\cong(V_N)_1$.

Then, since the rank of $V^{\orb(\amgis)}_1\cong(V_N)_1$ equals the central charge $24$, $V^{\orb(\amgis)}$ is isomorphic to a lattice \voa{} by \cite{DM04b}, and $V^{\orb(\amgis)}\cong V_N$ since the Niemeier lattices are uniquely determined by their root systems.
\end{proof}

We now consider the inverse orbifold automorphism $g\in\Aut(V_N)$ corresponding to one of the inner automorphisms $\amgis\in\Aut(V)$. In all cases we can show that $g$ must be algebraically conjugate to the automorphism $\bar g$ whose inverse orbifold automorphism $\bar\amgis$ was mimicked by $\amgis$.

\begin{lem}\label{lem:unique3}
All of the $226$ \good{} automorphisms $\bar{g}$ of the Niemeier lattice \voa{}s $V_N$ are uniquely specified up to algebraic conjugacy in $\Aut(V_N)$ by the following six properties:
\begin{enumerate}
\item\label{item:propunique1} the order~$n$ of $\bar{g}$,
\item\label{item:propunique2} $\bar{g}$ having type~$0$,
\item\label{item:propunique3} the extremality of $\bar{g}$,
\item\label{item:propunique4} the fixed-point Lie subalgebra $(V_N^{\bar{g}})_1$,
\item\label{item:propunique5} the orbifold Lie algebra $(V_N^{\orb(\bar{g})})_1$ and
\item\label{item:propunique6} the dimension of the eigenspace $\{v\in (V_N)_1\mid\bar{g}v=\e^{(2\pi\i)1/n}v\}$.
\end{enumerate}
\end{lem}
\begin{proof}
Let $\bar{g}$ be a \good{} automorphism of a Niemeier lattice \voa{} $V_N$ and let $g$ be an automorphism of the same \voa{} sharing properties \eqref{item:propunique1} to \eqref{item:propunique6} with $\bar{g}$. Note that $g$ is a \gdh{} by properties \eqref{item:propunique2} to \eqref{item:propunique5}.

\smallskip

First, we show that $g$ is again a \good{} automorphism. For this, by \autoref{thm:class}, it suffices to show that the order of the projection of $g$ to $\Aut(V_N)/K\cong H$ also has order~$n$. If $N$ is the Leech lattice~$\Lambda$, it was shown in \cite{MS23} that any \gdh{} must project to one of the $11$ conjugacy classes in $\Aut(V_\Lambda)/K\cong\O(\Lambda)$ listed in \autoref{table:11}. Since the rank of the fixed-point Lie subalgebra $(V_N^g)_1$ is determined by the projection to $H$, which moreover must have order dividing~$n$, the assertion follows by inspecting the $11$ possible Frame shapes.

Then, let $N$ be a Niemeier lattice other than the Leech lattice. Again, by inspecting the (up to) $25$ possible Frame shapes that can occur in $\Aut(V_N)/K\cong H$ (see \autoref{table:25} in the appendix), we conclude that the order of the projection of $g$ to $H$ must have order~$n$ except for possibly the Frame shape $3^8$ if $n=6$ and $\rk((V_N^g)_1)=8$ and $4^6$ if $n=8$ and $\rk((V_N^g)_1)=6$. By an explicit classification like in the proof of \autoref{thm:class} the existence of these spurious cases can be ruled out.

\smallskip

Knowing that $g$ is again a \good{} automorphism, we verify that every \good{} automorphism $\bar{g}$ is unique up to algebraic conjugacy amongst all the \good{} automorphisms with properties \eqref{item:propunique1} to \eqref{item:propunique6}. Hence, $g$ is algebraically conjugate to $\bar{g}$.
\end{proof}

Finally, we combine the above results to give a uniform proof of the uniqueness statement in \autoref{thm:classalt}:
\begin{thm}[Uniform Uniqueness]\label{thm:unique}
Let $\g\neq\{0\}$ be a Lie algebra on Schellekens' list. Then there is a Niemeier lattice $N$ and a \good{} automorphism $\bar{g}\in\Aut(V_N)$ such that any \strathol{} \voa{} $V$ of central charge $24$ with $V_1\cong\g$ satisfies $V\cong V_N^{\orb(\bar{g})}$. In particular, the \voa{} structure of $V$ is uniquely determined by the Lie algebra structure of $V_1$.
\end{thm}
In total we find $157$ such \good{} automorphisms $\bar{g}$, at least one for each of the $70$ non-zero Lie algebras $\g$.
\begin{proof}
By \autoref{prop:unique1} and \autoref{lem:unique2} there are $157$ inner automorphisms $\amgis$, at least one for each Lie algebra $\g$ and defined on any \strathol{} \voa{} $V$ of central charge $24$ with $V_1\cong\g$, whose orbifold constructions are isomorphic to a Niemeier lattice \voa{} $V_N$. Each $\amgis\in\Aut(V)$ was defined to resemble the inverse orbifold automorphism $\bar{\amgis}\in\Aut(V_N^{\orb(\bar{g})})$ corresponding to a \good{} automorphism $\bar{g}\in\Aut(V_N)$ with $(V_N^{\orb(\bar{g})})_1\cong\g$.

\smallskip

We then consider the inverse orbifold automorphism $g\in\Aut(V_N)$ corresponding to $\amgis$ so that $V_N^{\orb(g)}=V$. By construction of $\amgis$, the automorphism $g$ shares with the corresponding \good{} automorphism $\bar{g}\in\Aut(V_N)$ properties \eqref{item:propunique1} to \eqref{item:propunique6} listed in \autoref{lem:unique3}. Indeed, all of these properties follow from corresponding common properties of $\amgis$ and $\bar{\amgis}$.

We conclude that the so obtained $157$ automorphisms $g$ of the Niemeier lattice \voa{}s $V_N$ are algebraically conjugate to $\bar{g}$ in $\Aut(V_N)$ so that the \voa{} $V=V_N^{\orb(g)}$ is isomorphic to $V_N^{\orb(\bar{g})}$.
\end{proof}


\newpage

\section*{Appendix. List of \Good{} Orbifold Constructions}
For the reader's convenience, in \autoref{table:25} we list the algebraic conjugacy classes of outer automorphisms in $H=\O(N)/W$ for the Niemeier lattices $N$.

If $N$ is not the Leech lattice, we realise $H$ as the subgroup $H_\Delta\subseteq\O(N)$ of the lattice automorphisms fixing a choice $\Delta$ of the simple roots of the root system of $N$. Then the Frame shapes of the elements of $H_\Delta$ have only non-negative exponents and also describe how $H$ permutes the choice of simple roots. A total of $25$ different Frame shapes appear.

For the Leech lattice $\Lambda$ the root system is empty so that $H=\O(\Lambda)$. In total, $\O(\Lambda)$ has $70$ algebraic conjugacy classes $\nu$ with $\rk(\Lambda^\nu)>0$ but, for the sake of brevity, in the table we restrict to those automorphisms whose Frame shapes have only non-negative exponents. Then, the same $25$ Frame shapes appear (see \autoref{sec:leechrel}).

\medskip

\begin{table}[p]\caption{Algebraic conjugacy classes of outer automorphisms of the Niemeier lattices $N$. For the Leech lattice only those whose Frame shapes have only non-negative exponents are listed.}
\renewcommand{\arraystretch}{1.2}
\SMALL
\arraycolsep=0.2em
$
\begin{array}{lc|*{12}{r}r|r}
\multicolumn{2}{r|}{N} & \text{A}1 & \text{A}2 & \text{A}3 & \text{A}4 & \text{A}5 & \text{A}6 & \text{A}7 & \text{A}8 & \text{A}9 & \text{A}10 & \text{A}11 & \text{A}12 \\\cline{3-15}
\multicolumn{2}{l|}{\text{Frame Shp.}} & D_{24} & D_{16}E_{8} & E_{8}^3 & A_{24} & D_{12}^2 & A_{17}E_{7} & D_{10}E_{7}^2 & A_{15}D_{9} & D_{8}^3 & A_{12}^2 & A_{11}D_{7}E_{6} & E_{6}^4 \vphantom{h^{h^h}}\\\cline{1-15}
\text{A} & \sAA &  1  &  1  &  1  &  1  &  1  &  1  &  1  &  1  &  1  &  1  &  1  &  1  \vphantom{h^{h^h}}\\
\text{B} & \sBB &  .  &  .  &  1  &  .  &  .  &  1  &  1  &  1  &  1  &  .  &  1  &  2  \\
\text{C} & \sCC &  .  &  .  &  .  &  .  &  .  &  .  &  .  &  .  &  .  &  .  &  .  &  1  \\
\text{D} & \sDD &  .  &  .  &  .  &  1  &  1  &  .  &  .  &  .  &  .  &  1  &  .  &  .  \\%
\text{E} & \sEE &  .  &  .  &  .  &  .  &  .  &  .  &  .  &  .  &  .  &  .  &  .  &  .  \\
\text{F} & \sFF &  .  &  .  &  .  &  .  &  .  &  .  &  .  &  .  &  .  &  .  &  .  &  .  \\
\text{G} & \sGG &  .  &  .  &  .  &  .  &  .  &  .  &  .  &  .  &  .  &  .  &  .  &  1  \\
\text{H} & \sHH &  .  &  .  &  .  &  .  &  .  &  .  &  .  &  .  &  .  &  .  &  .  &  .  \\
\text{I} & \sII &  .  &  .  &  .  &  .  &  .  &  .  &  .  &  .  &  .  &  .  &  .  &  .  \\
\text{J} & \sJJ &  .  &  .  &  .  &  .  &  .  &  .  &  .  &  .  &  .  &  .  &  .  &  .  \\%
\text{K} & \sKK &  .  &  .  &  .  &  .  &  .  &  .  &  .  &  .  &  .  &  .  &  .  &  .  \\\cline{1-15}%
& 3^8           &  .  &  .  &  1  &  .  &  .  &  .  &  .  &  .  &  1  &  .  &  .  &  .  \vphantom{h^{h^h}}\\
& 2^44^4        &  .  &  .  &  .  &  .  &  .  &  .  &  .  &  .  &  .  &  .  &  .  &  1  \\
& 4^6           &  .  &  .  &  .  &  .  &  .  &  .  &  .  &  .  &  .  &  1  &  .  &  .  \\%
& 4^28^2        &  .  &  .  &  .  &  .  &  .  &  .  &  .  &  .  &  .  &  .  &  .  &  1  \\
& 1^211^2       &  .  &  .  &  .  &  .  &  .  &  .  &  .  &  .  &  .  &  .  &  .  &  .  \\
& 6^4           &  .  &  .  &  .  &  .  &  .  &  .  &  .  &  .  &  .  &  .  &  .  &  .  \\%
& 2^14^16^112^1 &  .  &  .  &  .  &  .  &  .  &  .  &  .  &  .  &  .  &  .  &  .  &  .  \\
& 1^12^17^114^1 &  .  &  .  &  .  &  .  &  .  &  .  &  .  &  .  &  .  &  .  &  .  &  .  \\
& 1^13^15^115^1 &  .  &  .  &  .  &  .  &  .  &  .  &  .  &  .  &  .  &  .  &  .  &  .  \\
& 3^121^1       &  .  &  .  &  .  &  .  &  .  &  .  &  .  &  .  &  .  &  .  &  .  &  .  \\
& 1^123^1       &  .  &  .  &  .  &  .  &  .  &  .  &  .  &  .  &  .  &  .  &  .  &  .  \\
& 12^2          &  .  &  .  &  .  &  .  &  .  &  .  &  .  &  .  &  .  &  .  &  .  &  .  \\%
& 4^120^1       &  .  &  .  &  .  &  .  &  .  &  .  &  .  &  .  &  .  &  .  &  .  &  .  \\%
& 2^122^1       &  .  &  .  &  .  &  .  &  .  &  .  &  .  &  .  &  .  &  .  &  .  &  .  \\\cline{1-15}%
\multicolumn{2}{l|}{\text{No.\ (A-K)}}
                &  1  &  1  &  2  &  2  &  2  &  2  &  2  &  2  &  2  &  2  &  2  &  5  \\
\multicolumn{2}{l|}{\text{No.\ (all)}}
                &  1  &  1  &  3  &  2  &  2  &  2  &  2  &  2  &  3  &  3  &  2  &  7  \\
\\[-5pt]
\multicolumn{2}{r|}{N} & \text{A}13 & \text{A}14 & \text{A}15 & \text{A}16 & \text{A}17 & \text{A}18 & \text{A}19 & \text{A}20 & \text{A}21 & \text{A}22 & \text{A}23 & \text{A}24 & \\\cline{3-15}
\multicolumn{2}{l|}{\text{Frame Shp.}} & A_{9}^2D_{6} & D_{6}^4 & A_{8}^3 & A_{7}^2D_{5}^2 & A_{6}^4 & A_{5}^4D_{4} & D_{4}^6 & A_{4}^6 & A_{3}^8 & A_{2}^{12} & A_{1}^{24} & \Lambda && \text{No.} \vphantom{h^{h^h}}\\\hline
\text{A} & \sAA &  1  &  1  &  1  &  1  &  1  &  1  &  1  &  1  &  1  &  1  &  1  &  1  && 24 \vphantom{h^{h^h}}\\
\text{B} & \sBB &  1  &  1  &  1  &  3  &  .  &  2  &  2  &  1  &  2  &  1  &  1  &  1  && 24 \\
\text{C} & \sCC &  .  &  1  &  .  &  .  &  1  &  1  &  2  &  .  &  1  &  1  &  1  &  1  && 10 \\
\text{D} & \sDD &  .  &  1  &  2  &  .  &  1  &  .  &  1  &  2  &  1  &  2  &  1  &  1  && 15 \\%
\text{E} & \sEE &  1  &  .  &  .  &  .  &  .  &  1  &  1  &  1  &  1  &  1  &  1  &  1  &&  8 \\
\text{F} & \sFF &  .  &  .  &  .  &  .  &  .  &  .  &  1  &  1  &  .  &  1  &  1  &  1  &&  5 \\
\text{G} & \sGG &  .  &  .  &  .  &  .  &  .  &  1  &  2  &  .  &  1  &  1  &  1  &  1  &&  8 \\
\text{H} & \sHH &  .  &  .  &  .  &  .  &  .  &  .  &  .  &  .  &  1  &  .  &  1  &  1  &&  3 \\
\text{I} & \sII &  .  &  .  &  .  &  .  &  .  &  1  &  .  &  .  &  .  &  1  &  1  &  1  &&  4 \\
\text{J} & \sJJ &  .  &  .  &  .  &  .  &  1  &  .  &  .  &  .  &  1  &  2  &  .  &  1  &&  5 \\%
\text{K} & \sKK &  .  &  .  &  .  &  .  &  .  &  .  &  .  &  1  &  .  &  1  &  1  &  1  &&  4 \\\hline%
& 3^8           &  .  &  .  &  1  &  .  &  .  &  .  &  1  &  1  &  .  &  1  &  1  &  1  &&  8 \vphantom{h^{h^h}}\\
& 2^44^4        &  .  &  .  &  .  &  1  &  .  &  .  &  1  &  .  &  1  &  1  &  1  &  1  &&  7 \\
& 4^6           &  .  &  1  &  .  &  .  &  1  &  .  &  .  &  1  &  1  &  1  &  1  &  1  &&  8 \\%
& 4^28^2        &  .  &  .  &  .  &  .  &  .  &  .  &  .  &  .  &  1  &  1  &  .  &  1  &&  4 \\
& 1^211^2       &  .  &  .  &  .  &  .  &  .  &  .  &  .  &  .  &  .  &  1  &  1  &  1  &&  3 \\
& 6^4           &  .  &  .  &  1  &  .  &  .  &  .  &  1  &  1  &  .  &  1  &  1  &  1  &&  6 \\%
& 2^14^16^112^1 &  .  &  .  &  .  &  .  &  .  &  .  &  1  &  .  &  .  &  .  &  1  &  1  &&  3 \\
& 1^12^17^114^1 &  .  &  .  &  .  &  .  &  .  &  .  &  .  &  .  &  1  &  .  &  1  &  1  &&  3 \\
& 1^13^15^115^1 &  .  &  .  &  .  &  .  &  .  &  .  &  1  &  .  &  .  &  .  &  1  &  1  &&  3 \\
& 3^121^1       &  .  &  .  &  .  &  .  &  .  &  .  &  .  &  .  &  .  &  .  &  1  &  1  &&  2 \\
& 1^123^1       &  .  &  .  &  .  &  .  &  .  &  .  &  .  &  .  &  .  &  .  &  1  &  1  &&  2 \\
& 12^2          &  .  &  .  &  .  &  .  &  .  &  .  &  .  &  1  &  .  &  1  &  1  &  1  &&  4 \\%
& 4^120^1       &  .  &  .  &  .  &  .  &  .  &  .  &  .  &  .  &  .  &  1  &  .  &  1  &&  2 \\%
& 2^122^1       &  .  &  .  &  .  &  .  &  .  &  .  &  .  &  .  &  .  &  1  &  .  &  1  &&  2 \\\hline%
\multicolumn{2}{l|}{\text{No.\ (A-K)}}
                &  3  &  4  &  4  &  4  &  4  &  7  & 10  &  7  &  9  & 12  & 10  & 11  &&110 \\
\multicolumn{2}{l|}{\text{No.\ (all)}}
                &  3  &  5  &  6  &  5  &  5  &  7  & 15  & 11  & 13  & 21  & 21  & 25  &&167 \\
\end{array}
$
\label{table:25}
\end{table}

\autoref{table:A} to \autoref{table:K} list the orbifold constructions associated with the $226$ \good{} automorphisms of the Niemeier lattice \voa{}s $V_N$ (see \autoref{thm:main}).

For each of the $11$ genera of orbit lattices $L_\g$, each corresponding to a conjugacy class in $\O(\Lambda)$ (see \autoref{table:11}), the rows represent the \strathol{} \voa{}s $V$ of central charge $24$ whose associated orbit lattice $L_\g$, $\g=V_1$, is in the selected genus (more than one \voa{} for a given isomorphism class $L_\g$ in genera D and J). We list the number of the corresponding entry in \cite{Hoe17} and \cite{Sch93} and the isomorphism type of the Lie algebra $\g=V_1$ with the levels $k_i$ of $\langle V_1\rangle\cong L_{\hat\g_1}(k_1,0)\otimes\cdots\otimes L_{\hat\g_s}(k_s,0)$ separated by a comma.

The columns represent the different algebraic conjugacy classes $\nu$ in the outer automorphism groups $\Aut(V_N)/K\cong H\subseteq\O(N)$ for all Niemeier lattices $N$ (denoted by the number of the corresponding entry in \cite{Hoe17} and by the root system) with the same Frame shape as the class in $\O(\Lambda)$.

The entry for a pair $(V,\nu)$ is the number of algebraic conjugacy classes of \good{} automorphisms projecting to $\nu$ such that $V_N^{\orb(g)}\cong V$.

\begin{table}[p]\caption{\Good{} orbifold constructions for genus A ($\sAA$).}
\SMALL
\arraycolsep=0.2em
$
\begin{array}{l|r|l|*{12}{r}}
\multicolumn{3}{r|}{V_N} & \text{A}1 & \text{A}2 & \text{A}3 & \text{A}4 & \text{A}5 & \text{A}6 & \text{A}7 & \text{A}8 & \text{A}9 & \text{A}10 & \text{A}11 & \text{A}12 \\\cline{4-15}
\multicolumn{3}{l|}{V_N^{\orb(g)}} & D_{24} & D_{16}E_{8} & E_{8}^3 & A_{24} & D_{12}^2 & A_{17}E_{7} & D_{10}E_{7}^2 & A_{15}D_{9} & D_{8}^3 & A_{12}^2 & \!A_{11}\!D_{7}\!E_{6}\! & E_{6}^4 \\\hline
\text{A}1  & 70 & D_{24,1}                   & 1 & . & . & . & . & . & . & . & . & . & . & . \\
\text{A}2  & 69 & D_{16,1}E_{8,1}            & . & 1 & . & . & . & . & . & . & . & . & . & . \\
\text{A}3  & 68 & E_{8,1}^3                  & . & . & 1 & . & . & . & . & . & . & . & . & . \\
\text{A}4  & 67 & A_{24,1}                   & . & . & . & 1 & . & . & . & . & . & . & . & . \\
\text{A}5  & 66 & D_{12,1}^2                 & . & . & . & . & 1 & . & . & . & . & . & . & . \\
\text{A}6  & 65 & A_{17,1}E_{7,1}            & . & . & . & . & . & 1 & . & . & . & . & . & . \\
\text{A}7  & 64 & D_{10,1}E_{7,1}^2          & . & . & . & . & . & . & 1 & . & . & . & . & . \\
\text{A}8  & 63 & A_{15,1}D_{9,1}            & . & . & . & . & . & . & . & 1 & . & . & . & . \\
\text{A}9  & 61 & D_{8,1}^3                  & . & . & . & . & . & . & . & . & 1 & . & . & . \\
\text{A}10 & 60 & A_{12,1}^2                 & . & . & . & . & . & . & . & . & . & 1 & . & . \\
\text{A}11 & 59 & A_{11,1}\!D_{7,1}\!E_{6,1} & . & . & . & . & . & . & . & . & . & . & 1 & . \\
\text{A}12 & 58 & E_{6,1}^4                  & . & . & . & . & . & . & . & . & . & . & . & 1 \\
\text{A}13 & 55 & A_{9,1}^2D_{6,1}           & . & . & . & . & . & . & . & . & . & . & . & . \\
\text{A}14 & 54 & D_{6,1}^4                  & . & . & . & . & . & . & . & . & . & . & . & . \\
\text{A}15 & 51 & A_{8,1}^3                  & . & . & . & . & . & . & . & . & . & . & . & . \\
\text{A}16 & 49 & A_{7,1}^2D_{5,1}^2         & . & . & . & . & . & . & . & . & . & . & . & . \\
\text{A}17 & 46 & A_{6,1}^4                  & . & . & . & . & . & . & . & . & . & . & . & . \\
\text{A}18 & 43 & A_{5,1}^4D_{4,1}           & . & . & . & . & . & . & . & . & . & . & . & . \\
\text{A}19 & 42 & D_{4,1}^6                  & . & . & . & . & . & . & . & . & . & . & . & . \\
\text{A}20 & 37 & A_{4,1}^6                  & . & . & . & . & . & . & . & . & . & . & . & . \\
\text{A}21 & 30 & A_{3,1}^8                  & . & . & . & . & . & . & . & . & . & . & . & . \\
\text{A}22 & 24 & A_{2,1}^{12}               & . & . & . & . & . & . & . & . & . & . & . & . \\
\text{A}23 & 15 & A_{1,1}^{24}               & . & . & . & . & . & . & . & . & . & . & . & . \\
\text{A}24 &  1 & \Lambda                    & . & . & . & . & . & . & . & . & . & . & . & . \\
\multicolumn{15}{c}{}\\
\multicolumn{3}{r|}{V_N} & \text{A}13 & \text{A}14 & \text{A}15 & \text{A}16 & \text{A}17 & \text{A}18 & \text{A}19 & \text{A}20 & \text{A}21 & \text{A}22 & \text{A}23 & \text{A}24 \\\cline{4-15}
\multicolumn{3}{l|}{V_N^{\orb(g)}} & A_{9}^2D_{6} & D_{6}^4 & A_{8}^3 & A_{7}^2D_{5}^2 & A_{6}^4 & A_{5}^4D_{4} & D_{4}^6 & A_{4}^6 & A_{3}^8 & A_{2}^{12} & A_{1}^{24} & \Lambda \\\hline
\text{A}1  & 70 & D_{24,1}                   & . & . & . & . & . & . & . & . & . & . & . & . \\
\text{A}2  & 69 & D_{16,1}E_{8,1}            & . & . & . & . & . & . & . & . & . & . & . & . \\
\text{A}3  & 68 & E_{8,1}^3                  & . & . & . & . & . & . & . & . & . & . & . & . \\
\text{A}4  & 67 & A_{24,1}                   & . & . & . & . & . & . & . & . & . & . & . & . \\
\text{A}5  & 66 & D_{12,1}^2                 & . & . & . & . & . & . & . & . & . & . & . & . \\
\text{A}6  & 65 & A_{17,1}E_{7,1}            & . & . & . & . & . & . & . & . & . & . & . & . \\
\text{A}7  & 64 & D_{10,1}E_{7,1}^2          & . & . & . & . & . & . & . & . & . & . & . & . \\
\text{A}8  & 63 & A_{15,1}D_{9,1}            & . & . & . & . & . & . & . & . & . & . & . & . \\
\text{A}9  & 61 & D_{8,1}^3                  & . & . & . & . & . & . & . & . & . & . & . & . \\
\text{A}10 & 60 & A_{12,1}^2                 & . & . & . & . & . & . & . & . & . & . & . & . \\
\text{A}11 & 59 & A_{11,1}\!D_{7,1}\!E_{6,1} & . & . & . & . & . & . & . & . & . & . & . & . \\
\text{A}12 & 58 & E_{6,1}^4                  & . & . & . & . & . & . & . & . & . & . & . & . \\
\text{A}13 & 55 & A_{9,1}^2D_{6,1}           & 1 & . & . & . & . & . & . & . & . & . & . & . \\
\text{A}14 & 54 & D_{6,1}^4                  & . & 1 & . & . & . & . & . & . & . & . & . & . \\
\text{A}15 & 51 & A_{8,1}^3                  & . & . & 1 & . & . & . & . & . & . & . & . & . \\
\text{A}16 & 49 & A_{7,1}^2D_{5,1}^2         & . & . & . & 1 & . & . & . & . & . & . & . & . \\
\text{A}17 & 46 & A_{6,1}^4                  & . & . & . & . & 1 & . & . & . & . & . & . & . \\
\text{A}18 & 43 & A_{5,1}^4D_{4,1}           & . & . & . & . & . & 1 & . & . & . & . & . & . \\
\text{A}19 & 42 & D_{4,1}^6                  & . & . & . & . & . & . & 1 & . & . & . & . & . \\
\text{A}20 & 37 & A_{4,1}^6                  & . & . & . & . & . & . & . & 1 & . & . & . & . \\
\text{A}21 & 30 & A_{3,1}^8                  & . & . & . & . & . & . & . & . & 1 & . & . & . \\
\text{A}22 & 24 & A_{2,1}^{12}               & . & . & . & . & . & . & . & . & . & 1 & . & . \\
\text{A}23 & 15 & A_{1,1}^{24}               & . & . & . & . & . & . & . & . & . & . & 1 & . \\
\text{A}24 &  1 & \Lambda                    & . & . & . & . & . & . & . & . & . & . & . & 1
\end{array}
$
\label{table:A}
\end{table}


\begin{table}\caption{\Good{} orbifold constructions for genus B ($\sBB$).}
\Small
\arraycolsep=0.15em
$
\begin{array}{l|r|l|*{12}{r}}
\multicolumn{3}{r|}{V_N} & \text{A}3 & \text{A}6 & \text{A}7 & \text{A}8 & \text{A}9 & \text{A}11 & \text{A}12 & \text{A}12 & \text{A}13 & \text{A}14 & \text{A}15 & \text{A}16 \\\cline{4-15}
\multicolumn{3}{l|}{V_N^{\orb(g)}}  & E_{8}^3 & A_{17}E_{7} & D_{10}E_{7}^2 & A_{15}D_{9} & D_{8}^3 & \!A_{11}\!D_{7}\!E_{6}\! & E_{6}^4 & E_{6}^4 & A_{9}^2D_{6} & D_{6}^4 & A_{8}^3 & A_{7}^2D_{5}^2 \\\hline
\text{B}1  & 62 & B_{8,1}E_{8,2}            & 1 & . & 1 & 1 & 1 & . & . & . & . & . & . & . \\
\text{B}2  & 56 & B_{6,1}C_{10,1}           & . & 1 & . & 1 & . & 1 & . & . & 1 & . & . & . \\
\text{B}3  & 52 & C_{8,1}F_{4,1}^2          & . & . & . & 1 & . & 1 & 1 & . & . & . & . & 1 \\
\text{B}4  & 53 & B_{5,1}E_{7,2}F_{4,1}     & . & . & 1 & . & . & 1 & . & 1 & . & 1 & . & . \\
\text{B}5  & 50 & A_{7,1}D_{9,2}            & . & 1 & . & . & 1 & . & . & . & . & . & 1 & . \\
\text{B}6  & 47 & B_{4,1}^2D_{8,2}          & . & . & . & 1 & 1 & . & . & . & . & 1 & . & 1 \\
\text{B}7  & 48 & B_{4,1}C_{6,1}^2          & . & . & . & . & . & 1 & . & . & 1 & . & . & 1 \\
\text{B}8  & 44 & A_{5,1}C_{5,1}E_{6,2}     & . & . & . & . & . & . & . & 1 & 1 & . & . & . \\
\text{B}9  & 40 & A_{4,1}A_{9,2}B_{3,1}     & . & . & . & . & . & . & . & . & . & . & 1 & . \\
\text{B}10 & 39 & B_{3,1}^2C_{4,1}D_{6,2}   & . & . & . & . & . & 1 & . & . & . & 1 & . & 1 \\
\text{B}11 & 38 & C_{4,1}^4                 & . & . & . & . & . & . & 1 & . & . & . & . & 1 \\
\text{B}12 & 33 & A_{3,1}A_{7,2}C_{3,1}^2   & . & . & . & . & . & . & . & . & . & . & . & . \\
\text{B}13 & 31 & A_{3,1}^2D_{5,2}^2        & . & . & . & . & . & . & . & . & 1 & . & . & . \\
\text{B}14 & 26 & A_{2,1}^2A_{5,2}^2B_{2,1} & . & . & . & . & . & . & . & . & . & . & . & . \\
\text{B}15 & 25 & B_{2,1}^4D_{4,2}^2        & . & . & . & . & . & . & . & . & . & . & . & 1 \\
\text{B}16 & 16 & A_{1,1}^4A_{3,2}^4        & . & . & . & . & . & . & . & . & . & . & . & . \\
\text{B}17 &  5 & A_{1,2}^{16}              & . & . & . & . & . & . & . & . & . & . & . & . \\
\multicolumn{15}{c}{}\\
\multicolumn{3}{r|}{V_N} & \text{A}16 & \text{A}16 & \text{A}18 & \text{A}18 & \text{A}19 & \text{A}19 & \text{A}20 & \text{A}21 & \text{A}21 & \text{A}22 & \text{A}23 & \text{A}24 \\\cline{4-15}
\multicolumn{3}{l|}{V_N^{\orb(g)}}  & A_{7}^2D_{5}^2 & A_{7}^2D_{5}^2 & A_{5}^4D_{4} & A_{5}^4D_{4} & D_{4}^6 & D_{4}^6 & A_{4}^6 & A_{3}^8 & A_{3}^8 & A_{2}^{12} & A_{1}^{24} & \Lambda \\\hline
\text{B}1  & 62 & B_{8,1}E_{8,2}            & . & . & . & . & . & . & . & . & . & . & . & . \\
\text{B}2  & 56 & B_{6,1}C_{10,1}           & . & . & . & . & . & . & . & . & . & . & . & . \\
\text{B}3  & 52 & C_{8,1}F_{4,1}^2          & . & . & . & . & . & . & . & . & . & . & . & . \\
\text{B}4  & 53 & B_{5,1}E_{7,2}F_{4,1}     & . & 1 & . & . & . & . & . & . & . & . & . & . \\
\text{B}5  & 50 & A_{7,1}D_{9,2}            & 1 & . & . & . & . & . & . & . & . & . & . & . \\
\text{B}6  & 47 & B_{4,1}^2D_{8,2}          & . & 1 & . & . & . & 1 & . & . & . & . & . & . \\
\text{B}7  & 48 & B_{4,1}C_{6,1}^2          & . & . & 1 & . & . & . & . & . & . & . & . & . \\
\text{B}8  & 44 & A_{5,1}C_{5,1}E_{6,2}     & 1 & . & . & 1 & . & . & . & . & . & . & . & . \\
\text{B}9  & 40 & A_{4,1}A_{9,2}B_{3,1}     & . & 1 & . & 1 & . & . & 1 & . & . & . & . & . \\
\text{B}10 & 39 & B_{3,1}^2C_{4,1}D_{6,2}   & 1 & . & . & 1 & 1 & . & . & . & 1 & . & . & . \\
\text{B}11 & 38 & C_{4,1}^4                 & . & . & 1 & . & . & . & . & 1 & . & . & . & . \\
\text{B}12 & 33 & A_{3,1}A_{7,2}C_{3,1}^2   & . & 1 & 1 & 1 & . & . & . & . & 1 & . & . & . \\
\text{B}13 & 31 & A_{3,1}^2D_{5,2}^2        & 1 & . & . & . & . & 1 & 1 & . & 1 & . & . & . \\
\text{B}14 & 26 & A_{2,1}^2A_{5,2}^2B_{2,1} & . & . & . & 1 & . & . & 1 & . & 1 & 1 & . & . \\
\text{B}15 & 25 & B_{2,1}^4D_{4,2}^2        & . & . & . & . & 1 & 1 & . & 1 & 1 & . & 1 & . \\
\text{B}16 & 16 & A_{1,1}^4A_{3,2}^4        & . & . & 1 & . & . & . & . & . & 1 & 1 & 1 & . \\
\text{B}17 &  5 & A_{1,2}^{16}              & . & . & . & . & . & . & . & 1 & . & . & 1 & 1
\end{array}
$
\end{table}


\begin{table}\caption{\vspace{-0.2cm}\Good{} orbifold constructions for genus C ($\sCC$).}
\arraycolsep=0.3em
$
\begin{array}{l|r|l|*{10}{r}}
\multicolumn{3}{r|}{V_N} & \text{A}12 & \text{A}14 & \text{A}17 & \text{A}18 & \text{A}19 & \text{A}19 & \text{A}21 & \text{A}22 & \text{A}23 & \text{A}24 \\\cline{4-13}
\multicolumn{3}{l|}{V_N^{\orb(g)}}  & E_{6}^4 & D_{6}^4 & A_{6}^4 & A_{5}^4D_{4} & D_{4}^6 & D_{4}^6 & A_{3}^8 & A_{2}^{12} & A_{1}^{24} & \Lambda \\\hline
\text{C}1 & 45 & A_{5,1}E_{7,3}          & 1 & 1 & 1 & 1 & . & . & . & . & . & . \\
\text{C}2 & 34 & A_{3,1}D_{7,3}G_{2,1}   & . & 1 & 1 & 1 & . & 1 & 1 & . & . & . \\
\text{C}3 & 32 & E_{6,3}G_{2,1}^3        & 1 & . & . & 1 & 1 & 1 & . & 1 & . & . \\
\text{C}4 & 27 & A_{2,1}^2A_{8,3}        & . & . & 1 & 1 & . & . & 1 & 1 & . & . \\
\text{C}5 & 17 & A_{1,1}^3A_{5,3}D_{4,3} & . & . & . & 1 & . & 1 & 1 & 1 & 1 & . \\
\text{C}6 &  6 & A_{2,3}^6               & . & . & . & . & 1 & . & . & 1 & 1 & 1
\end{array}
$
\vspace{0.7cm}
\end{table}


\begin{table}\caption{\vspace{-0.2cm}\Good{} orbifold constructions for genus D ($\sDD$).}
\Small
\arraycolsep=0.2em
$
\begin{array}{l|r|l|rrrrrr|rrrrrr|rrr}
\multicolumn{3}{r|}{V_N} & \text{A}4 & \text{A}10 & \text{A}15 & \text{A}17 & \text{A}20 & \text{A}22 & \text{A}5 & \text{A}14 & \text{A}19 & \text{A}21 & \text{A}23 & \text{A}24 & \text{A}15 & \text{A}20 & \text{A}22  \\\cline{4-18}
\multicolumn{3}{l|}{V_N^{\orb(g)}} & A_{24} & A_{12}^2 & A_8^3 & A_6^4 & A_4^6 & A_2^{12} & D_{12}^2 & D_6^4 & D_4^6 & A_3^8 & A_1^{24} & \Lambda & A_8^3 & A_4^6 & A_2^{12} \\\hline
\text{D1a} & 57 & B_{12,2}         & 1 & . & . & . & . & . & 1 & . & . & . & . & . & . & . & . \\
\text{D1b} & 41 & B_{6,2}^2        & . & 1 & . & . & . & . & . & 1 & . & . & . & . & . & . & . \\
\text{D1c} & 29 & B_{4,2}^3        & . & . & 1 & . & . & . & . & . & 1 & . & . & . & . & . & . \\
\text{D1d} & 23 & B_{3,2}^4        & . & . & . & 1 & . & . & . & . & . & 1 & . & . & . & . & . \\
\text{D1e} & 12 & B_{2,2}^6        & . & . & . & . & 1 & . & . & . & . & . & 1 & . & . & . & . \\
\text{D1f} &  2 & A_{1,4}^{12}     & . & . & . & . & . & 1 & . & . & . & . & . & 1 & . & . & . \\\hline
\text{D2a} & 36 & A_{8,2}F_{4,2}   & . & . & . & . & . & . & . & . & . & . & . & . & 1 & . & . \\
\text{D2b} & 22 & A_{4,2}^2C_{4,2} & . & . & . & . & . & . & . & . & . & . & . & . & . & 1 & . \\
\text{D2c} & 13 & A_{2,2}^4D_{4,4} & . & . & . & . & . & . & . & . & . & . & . & . & . & . & 1
\end{array}
$
\vspace{0.7cm}
\end{table}


\begin{table}\caption{\vspace{-0.2cm}\Good{} orbifold constructions for genus E ($\sEE$).}
$
\begin{array}{l|r|l|*{8}{r}}
\multicolumn{3}{r|}{V_N} & \text{A}13 & \text{A}18 & \text{A}19 & \text{A}20 & \text{A}21 & \text{A}22 & \text{A}23 & \text{A}24 \\\cline{4-11}
\multicolumn{3}{l|}{V_N^{\orb(g)}} & A_9^2D_6 & A_5^4D_4 & D_4^6 & A_4^6 & A_3^8 & A_2^{12} & A_1^{24} & \Lambda \\\hline
\text{E}1 & 35 & A_{3,1}C_{7,2}          & 1 & 2 & . & 1 & 1 & . & . & . \\
\text{E}2 & 28 & A_{2,1}B_{2,1}E_{6,4}   & 1 & . & 1 & 2 & 1 & 1 & . & . \\
\text{E}3 & 18 & A_{1,1}^3A_{7,4}        & . & 1 & . & 1 & 1 & 2 & 1 & . \\
\text{E}4 & 19 & A_{1,1}^2C_{3,2}D_{5,4} & . & 1 & 1 & 1 & 3 & 1 & 2 & . \\
\text{E}5 &  7 & A_{1,2}A_{3,4}^3        & . & . & . & . & 1 & 1 & 2 & 1
\end{array}
$
\vspace{0.7cm}
\end{table}


\begin{table}\caption{\vspace{-0.2cm}\Good{} orbifold constructions for genus F ($\sFF$).}
$
\begin{array}{l|r|l|*{5}{r}}
\multicolumn{3}{r|}{V_N} & \text{A}19 & \text{A}20 & \text{A}22 & \text{A}23 & \text{A}24 \\\cline{4-8}
\multicolumn{3}{l|}{V_N^{\orb(g)}}  & D_4^6 & A_4^6 & A_2^{12} & A_1^{24} & \Lambda \\\hline
\text{F}1 & 20 & A_{1,1}^2D_{6,5} & 1 & 1 & 1 & 1 & . \\
\text{F}2 &  9 & A_{4,5}^2        & . & 1 & 1 & 1 & 1
\end{array}
$
\vspace{0.7cm}
\end{table}


\begin{table}\caption{\vspace{-0.2cm}\Good{} orbifold constructions for genus G ($\sGG$).}
$
\begin{array}{l|r|l|*{8}{r}}
\multicolumn{3}{r|}{V_N} & \text{A}12 & \text{A}18 & \text{A}19 & \text{A}19 & \text{A}21 & \text{A}22 & \text{A}23 & \text{A}24 \\\cline{4-11}
\multicolumn{3}{l|}{V_N^{\orb(g)}}  & E_6^4 & A_5^4D_4 & D_4^6 & D_4^6 & A_3^8 & A_2^{12} & A_1^{24} & \Lambda \\\hline
\text{G}1 & 21 & A_{1,1}C_{5,3}G_{2,2} & 1 & 3 & 2 & 1 & 2 & 2 & 2 & . \\
\text{G}2 &  8 & A_{1,2}A_{5,6}B_{2,3} & . & 1 & 1 & 1 & 2 & 2 & 5 & 1
\end{array}
$
\vspace{0.7cm}
\end{table}


\begin{table}\caption{\vspace{-0.2cm}\Good{} orbifold constructions for genus H ($\sHH$).}
$
\begin{array}{l|r|l|*{3}{r}}
\multicolumn{3}{r|}{V_N} & \text{A}21 & \text{A}23 & \text{A}24 \\\cline{4-6}
\multicolumn{3}{l|}{V_N^{\orb(g)}}  & A_3^8 & A_1^{24} & \Lambda \\\hline
\text{H}1 & 11 & A_{6,7} & 1 & 1 & 1
\end{array}
$
\vspace{0.7cm}
\end{table}


\begin{table}\caption{\vspace{-0.2cm}\Good{} orbifold constructions for genus I ($\sII$).}
$
\begin{array}{l|r|l|*{4}{r}}
\multicolumn{3}{r|}{V_N} & \text{A}18 & \text{A}22 & \text{A}23 & \text{A}24 \\\cline{4-7}
\multicolumn{3}{l|}{V_N^{\orb(g)}}  & A_5^4D_4 & A_2^{12} & A_1^{24} & \Lambda \\\hline
\text{I}1 & 10 & A_{1,2}D_{5,8} & 1 & 2 & 2 & 1
\end{array}
$
\vspace{0.7cm}
\end{table}


\begin{table}\caption{\vspace{-0.2cm}\Good{} orbifold constructions for genus J ($\sJJ$).}
$
\begin{array}{l|r|l|rr|rr|r}
\multicolumn{3}{r|}{V_N} & \text{A}17 & \text{A}22 & \text{A}21 & \text{A}24 & \text{A}22 \\\cline{4-8}
\multicolumn{3}{l|}{V_N^{\orb(g)}}  & A_6^4 & A_2^{12} & A_3^8 & \Lambda & A_2^{12} \\\hline
\text{J1a} & 14 & A_{2,2}F_{4,6}  & 1 & . & 1 & . & 1 \\
\text{J1b} &  3 & A_{2,6}D_{4,12} & . & 1 & . & 1 & 1
\end{array}
$
\vspace{0.7cm}
\end{table}


\begin{table}\caption{\vspace{-0.2cm}\Good{} orbifold constructions for genus K ($\sKK$).}
$
\begin{array}{l|r|l|rr|rr}
\multicolumn{3}{r|}{V_N} & \text{A}20 & \text{A}22 & \text{A}23 & \text{A}24 \\\cline{4-7}
\multicolumn{3}{l|}{V_N^{\orb(g)}}& A_4^6 & A_2^{12} & A_1^{24} & \Lambda \\\hline
\text{K}1 & 4 & C_{4,10} & 1 & 1 & 1 & 1
\end{array}
$
\label{table:K}
\vspace{0.7cm}
\end{table}

\FloatBarrier


The following five tables list the non-trivial powers of the $226$ \good{} automorphisms $g$ of the Niemeier lattice \voa{}s $V_N$.

\smallskip

For the five genera E, G, I, J and K of orbit lattices corresponding to a Frame shape of an element of composite order we list the non-trivial powers $g^p$ with prime exponent $p$.

The rows and columns are labelled as in \autoref{table:A} to \autoref{table:K}. An entry $Xi_1,\ldots,i_k$ for a power $g^p$ refers to the elements in the table for the genus associated with $g^p$ in the column corresponding to the same Niemeier lattice and the rows labelled by the orbifold constructions $Xi_1,\ldots,Xi_k$.

If the same Niemeier lattice $N$ is the label for more than one column for a power $g^p$, the extra row named $g^p$ identifies the column
if there is a possible ambiguity.

Lastly, if there is more than one element $g^p$ for an entry specified by $Xi_\nu$, an extra index $l$ refers to the $l$-th element in that entry (this applies to two cases of $g$ for genus I).

\begin{table}\caption{Powers $g^2$ for genus E ($\sEE$).}
\begin{tabular}{l|*{8}{l}}
& A13 & A18 & A19 & A20 & A21 & A22 & A23 & A24 \\\cline{2-9}
& $A_9^2D_6$ & $A_5^4D_4$ & $D_4^6$ & $A_4^6$ & $A_3^8$ & $A_2^{12}$ & $A_1^{24}$ & $\Lambda$ \\\hline
$g^2$&  .  &  1    & 2   & .      &    2      & .      & .      & . \\\hline
E1   & B7  & B7,12 & .   & B9     & B12       & .      & .      & . \\
E2   & B13 & .     & B13 & B13,14 & B13       & B14    & .      & . \\
E3   & .   & B16   & .   & B14    & B16       & B14,16 & B16    & . \\
E4   & .   & B12   & B15 & B9     & B12,15,16 & B14    & B15,16 & . \\
E5   & .   & .     & .   & .      & B16       & B14    & B16,17 & B17
\end{tabular}
\label{table:Epower}
\vspace{0.7cm}
\end{table}


\begin{table}\caption{Powers $g^2$ and $g^3$ for genus G ($\sGG$).}
\arraycolsep=0.4em
\begin{tabular}{l|*{8}{l}}
& A12 & A18 & A19 & A19 & A21 & A22 & A23 & A24 \\\cline{2-9}
& $E_6^4$ & $A_5^4D_4$ & $D_4^6$ & $D_4^6$ & $A_3^8$ & $A_2^{12}$ & $A_1^{24}$ & $\Lambda$ \\\hline
$g^2$ &  .  &  .        &  2     & 1   & .     & .      & .  & .               \\
$g^3$ &  1  &  1        &  1     & 2   & 1     & .      & .  & .               \\\hline
G1    & C3  & C3,3,5    & C3,5   & C3  & C5,5  & C3,5   & C5,5            & .  \\
      & B11 & B11,16,11 & B15,15 & B15 & B11,15& B16,16 & B15,16          & .  \\[1.6mm]
G2    &  .  & C5        & C5     & C6  & C5,5  & C5,6   & C5,5,6,6,6      & C6 \\
      &  .  & B16       & B15    & B15 & B15,17& B16,16 & B16,17,15,16,17 & B17
\end{tabular}
\label{table:Gpower}
\vspace{0.7cm}
\end{table}


\begin{table}\caption{Powers $g^2$ for genus I ($\sII$).}
\begin{tabular}{l|*{4}{l}}
& A18 & A22 & A23 & A24 \\\cline{2-5}
& $A_5^4D_4$ & $A_2^{12}$ & $A_1^{24}$ & $\Lambda$ \\\hline
I1 & E3 & E3$_2$,5 & E3,5$_2$ & E5
\end{tabular}
\label{table:Ipower}
\vspace{0.7cm}
\end{table}


\begin{table}\caption{Powers $g^2$ and $g^3$ for genus J ($\sJJ$).}
\begin{tabular}{l|ll|ll|l}
& A17 & A22 & A21 & A24 & A22 \\\cline{2-6}
& $A_6^4$ & $A_2^{12}$ & $A_3^8$ & $\Lambda$ & $A_2^{12}$ \\\hline
J1a  & C4  & .   & C4  & .   & C4  \\
     & D1d & .   & D1d & .   & D2c \\[1.6mm]
J1b  & .   & C6  & .   & C6  & C6  \\
     & .   & D1f & .   & D1f & D2c
\end{tabular}
\label{table:Jpower}
\vspace{0.7cm}
\end{table}


\begin{table}\caption{Powers $g^2$ and $g^5$ for genus K ($\sKK$).}
\begin{tabular}{l|ll|ll}
& A20 & A22 & A23 & A24 \\\cline{2-5}
& $A_4^6$ & $A_2^{12}$ & $A_1^{24}$ & $\Lambda$ \\\hline
K1 & F2  & F2  & F2  & F2  \\
   & D1e & D1f & D1e & D1f
\end{tabular}
\label{table:Kpower}
\vspace{0.7cm}
\end{table}

\FloatBarrier


\bibliographystyle{alpha_noseriescomma}
\bibliography{quellen}{}

\end{document}